\input ./style/arxiv-vmsta.cfg
\documentclass[numbers,compress,v1.0.1]{vmsta}

\volume{4}
\issue{2}
\pubyear{2017}
\firstpage{161}
\lastpage{188}
\doi{10.15559/17-VMSTA79}

\startlocaldefs

\newcommand{\rrvert}{\vert}
\newcommand{\llvert}{\vert}
\allowdisplaybreaks

\newtheorem{theorem}{Theorem}
\newtheorem{lemma}{Lemma}

\theoremstyle{definition}

\newtheorem{remark}{Remark}

\endlocaldefs

\begin{document}
\begin{frontmatter}

\title{Compositions of Poisson and Gamma processes}

%

\author{\inits{K.}\fnm{Khrystyna}\snm{Buchak}}\email{kristina.kobilich@gmail.com}
\author{\inits{L.}\fnm{Lyudmyla}\snm{Sakhno}\corref{cor1}}\email{lms@univ.kiev.ua}
\cortext[cor1]{Corresponding author.}

\address{Mechanics and Mathematics Faculty,\\
Taras Shevchenko National University of Kyiv\\
Volodymyrska 64/11, 01601 Kyiv, Ukraine}



\markboth{K. Buchak, L. Sakhno}{Compositions of Poisson and Gamma processes}



\begin{abstract}
In the paper we study the models of time-changed Poisson and Skellam-type
processes, where the role of time is played by compound Poisson-Gamma
subordinators and their inverse (or first passage time) processes. We obtain
explicitly the probability distributions of considered time-chan\-ged
processes and discuss their properties.
\end{abstract}


\received{9 February 2017}
\revised{16 May 2017}
\accepted{2 June 2017}
\publishedonline{29 June 2017}
\end{frontmatter}

\section{Introduction}

Stochastic processes with random time and more general compositions of
pro\-cesses are quite popular topics of recent studies both in the
theory of
stochastic processes and in various applied areas. Specifically, in
financial mathematics, models with random clock (or time change) allow to
capture more realistically the relationship between calendar time and
financial markets activity. Models with random time appear in reliability
and queuing theory, biological, ecological and medical research, note also
that for solving some problems of statistical estimation sampling of a
stochastic process at random times or on the trajectory of another process
can be used. Some examples of applications are described, for example, in
\cite{D}.

In the present paper we study various compositions of Poisson and Gamma
processes. We only consider the cases when processes used for compositions
are independent.

The Poisson process directed by a Gamma process and, in reverse,
the Gamma process directed by a Poisson process one can encounter,
for example, in the book by W. Feller \cite{F}, where distributions
of these processes are presented as particular examples within the
general framework of directed processes.

Time-changed Poisson processes have been investigated extensively in the
literature. We mention, for example, the recent comprehensive study
undertaken in \cite{OT} concerned with the processes $N(H^{f}(t))$,
where $%
N(t)$ is a Poisson process and $H^{f}(t)$ is an arbitrary
subordinator with Laplace exponent $f$, independent of $N(t)$. The
particular emphasis has been made on the cases where
$f(u)=u^{\alpha}$, $\alpha\in(0,1)$ (space-fractional Poisson
process), $f(u)= ( u+\theta ) ^{\alpha}-\theta^{\alpha
}$, $\alpha\in(0,1)$,\ $\theta>0\ $(tempered Poisson process)
and $f(\lambda)=\log(1+\lambda)$ (negative binomial process),
in the last case we have the Poisson process with Gamma
subordinator.

The most intensively studied models of Poisson processes with
time-change\ are two fractional extensions of Poisson process,
namely, the space-fractional and the time-fractional Poisson
processes, obtained by choosing a stable subordinator or its
inverse process in the role of time correspondingly; the
literature devoted to these processes is rather voluminous, we
refer, e.g., to the recent papers \cite{GOS,OP,KNV} (and references therein), and to the paper \cite{LMSS},
where the correlation structure of the time-changed L\'{e}vy
processes has been investigated and the correlation of
time-fractional Poisson process was discussed among the variety of
other examples. The most recent results are concerned with
non-homogeneous fractional Poisson processes (see, e.g.
\cite{Leonenko2017} and references therein).

Interesting models of processes are based on the use of the difference of
two Poisson processes, so-called Skellam processes, and their
generalizations via time change. Investigation of these models can be found,
for example, in \cite{B-N}, and we also refer to the paper \cite{KLS}, where
fractional Skellam processes were introduced and studied.\looseness=1

In the present paper we study time-changed Poisson and Skellam
processes, where the role of time is played by compound
Poisson-Gamma subordinators and their inverse (or first passage
time) processes. Some motivation for our study is presented in
Remarks \ref{remark1} and \ref{remark2} in Section~\ref{sec:com P-G}.\looseness=0

We obtain explicitly the probability distributions of considered
time-changed processes and their first and second order moments.

In particular, for the case, where time-change is taken by means
of compound Poisson-exponential subordinator and its inverse
process, corresponding probability distributions of time-changed
Poisson and Skellam processes are presented in terms of
generalized Mittag-Leffler functions.

We also find the relation, in the form of differential equation,
between the
distribution of Poisson process, time-changed by Poisson-exponential
process, and the distribution of Poisson processes time-changed by inverse
Poisson-exponential process.

The paper is organized as follows. In Section~\ref{sec:prel} we recall definitions of
processes which will be considered in the paper, in Section~\ref{sec:com P-G} we
discuss the
main features of the compound Poisson-Gamma process $G_{N}(t)$. In
Section~\ref{sec:com}
we study Poisson and Skellam-type processes time-changed by the process
$%
G_{N}(t)$; in Section~\ref{section5} we investigate time-changed Poisson and
Skellam-type processes, where the role of time is played by the
inverse process for $G_{N}(t)$, we also discuss some properties of
the inverse processes. Appendix contains details of the derivation
of some results stated in Section~\ref{section5}.

\section{Preliminaries}
\label{sec:prel}

In this section we recall definitions of processes, which will be
considered in the paper (see, e.g. \cite{A,B}).

The Poisson process $N(t)$ with intensity parameter $\lambda>0$ is a
L\'{e}vy process with values in $N\cup\{0\}$ such that each $N(t)$ has
Poisson distribution with parameter $\lambda t$, that is,
\[
P\bigl\{N(t)=n\bigr\}=\frac{ ( \lambda t ) ^{n}}{n!}e^{-\lambda t};
\]
L\'{e}vy measure of $N(t)$ can be written in the form $\nu(du)=\lambda
\delta_{\{1\}}(u)du$, where $\delta_{\{1\}}(u)$ is the Dirac delta
centered at $u=1$, the corresponding Bern\v{s}tein function (or Laplace
exponent of $%
N(t)$) is $f(u)=\lambda ( 1-e^{-u} ), u>0$.

The Gamma process $G(t)$ with parameters $\alpha,\beta>0$ is a L\'{e}vy
process such that each $G(t)$ follows Gamma distribution $\varGamma(\alpha
t,\beta)$, that is, has the density
\[
h_{G(t)}(x)=\frac{\beta^{\alpha t}}{\varGamma(\alpha t)}x^{\alpha
t-1}e^{-\beta x},\quad x\geq0;
\]
The L\'{e}vy measure of $G(t)$ is $\nu(du)=\alpha u^{-1}e^{-\beta
u}du$ and the corresponding Bern\v{s}tein function is
\[
f(u)=\alpha\log \biggl( 1+\frac{u}{\beta} \biggr) .
\]
%

The Skellam process is defined as
\[
S(t)=N_{1}(t)-N_{2}(t),\quad t\geq0,
\]
where $N_{1}(t),t\geq0$, and $N_{2}(t),t\geq0$, are two independent Poisson
processes with intensities $\lambda_{{1}}>0$ and $\lambda_{2}>0$,
respectively.

The probability mass function of $S(t)$ is given by
\begin{align}
&s_{k}(t)=P \bigl( S(t)=k \bigr) =e^{-t ( \lambda_{{1}}+\lambda
_{2} ) } \biggl(
\frac{\lambda_{{1}}}{\lambda_{2}} \biggr) ^{k/2}I_{|k|} ( 2t\sqrt{
\lambda_{{1}}\lambda_{2}} ) , \notag\\
&\quad k\in\Bbb{Z}=\{0,\pm1,\pm2,\ldots\},\label{skellamdistr}
\end{align}
where $I_{k}$ is the modified Bessel function of the first kind
\cite{Sn}:
\begin{equation}
I_{k}(z)=\sum_{n=0}^{\infty}
\frac{ ( z/2 ) ^{2n+k}}{n!(n+k)!}. \label{Bessel}
\end{equation}
The Skellam process is a L\'{e}vy process, its L\'{e}vy measure is
the linear combination of two Dirac measures: $\nu(du)=\lambda
_{1}\delta_{\{1\}}(u)du+\lambda_{2}\delta_{\{-1\}}(du)$, 
the corresponding Bern\v{s}tein function is
\[
f_{S}(\theta)=\int_{-\infty}^{\infty} \bigl(
1-e^{-\theta y} \bigr) \nu (dy)=\lambda_{1} \bigl(
1-e^{-\theta} \bigr) +\lambda_{2} \bigl( 1-e^{\theta}
\bigr) ,
\]
and the moment generating function is given by
\[
\mathsf{E}e^{\theta S(t)}=e^{-t ( \lambda_{{1}}+\lambda
_{2}-\lambda_{{%
1}}e^{\theta}-\lambda_{2}e^{-\theta} ) },\quad \theta\in
\Bbb{R}%
.
\]

Skellam processes are considered, for example, in \cite{B-N}, the Skellam
distribution had been introduced and studied in \cite{Sk} and \cite{Ir}.

We will consider Skellam processes with time change
\[
S_{_{I}}(t)=S\bigl(X(t)\bigr)=N_{1}\bigl(X(t)
\bigr)-N_{2}\bigl(X(t)\bigr),
\]
where $X(t)$ is a subordinator independent of $N_{1}(t)$ and
$N_{2}(t)$, and
will call such processes 
time-changed Skellam processes of type I. We
will also consider the processes of the form
\[
S_{\mathit{II}}(t)=N_{1}\bigl(X_{1}(t)
\bigr)-N_{2}\bigl(X_{2}(t)\bigr),
\]
where $N_{1}(t)$, $N_{2}(t)$ are two independent Poisson processes with
intensities $\lambda_{1}>0$ and $\lambda_{2}>0$, and $X_{1}(t)$, $X_{2}(t)$
are two independent copies of a subordinator $X(t)$, which are also
independent of $N_{1}(t)$ and $N_{2}(t)$, and we will call the process $
S_{\mathit{II}}(t)$ a time-changed Skellam process of type II.

To represent distributions and other characteristics of processes considered
in the next sections, we will use some special functions, besides the
modified Bessel function introduced above. Namely, we will use the Wright
function
\begin{equation}
\varPhi(\rho,\delta,z)=\sum_{k=0}^{\infty}
\frac{z^{k}}{k!\varGamma(\rho
k+\delta)},\quad z\in C,\ \rho\in(-1,0)\cup(0,\infty ),\ \delta\in C, \label{Wr}
\end{equation}
for $\delta=0$, \eqref{Wr} is simplified as follows:
\begin{equation}
\varPhi(\rho, 0 ,z)=\sum_{k=1}^{\infty}
\frac{z^{k}}{k!\varGamma(\rho
k)}; \label{Wr0}
\end{equation}
the two-parameter generalized Mittag-Leffler function
\begin{equation}
\mathcal{E}_{\rho,\delta}(z)=\sum_{k=0}^{\infty}
\frac{z^{k}}{\varGamma
(\rho
k+\delta)},\quad z\in C,\ \rho\in(0,\infty), \delta \in (0,
\infty); \label{ML2}
\end{equation}
and the three-parameter generalized Mittag-Leffler function
\begin{equation}
\mathcal{E}_{\rho,\delta}^{\gamma}(z)=\sum_{k=0}^{\infty}
\frac
{\varGamma
(\gamma+k)}{\varGamma(\gamma)}\frac{z^{k}}{k!\varGamma(\rho k+\delta
)},\quad z\in C,\ \rho,\delta,\gamma\in
C, \label{ML3}
\end{equation}
with $\text{ Re}(\rho)>0,\text{Re}(\delta)>0,\text{Re}(\gamma
)>0$ (see, e.g., \cite{Haubold} for definitions and properties of
these functions).

\section{Compound Poisson-Gamma process}
\label{sec:com P-G} The first example of compositions of Poisson
and Gamma processes which we consider in the paper is the compound
Poisson-Gamma process. This is a well known process, however here
we would like to focus on some of its important features.

Let $N(t)$ be a Poisson process and $ \{ G_{n},n\geq1 \}
$ be a sequence of i.i.d. Gamma random variables independent of
$N(t)$. Then compound Poisson process with Gamma distributed jumps
is defined as
\[
Y(t)=\sum_{n=1}^{N(t)}G_{n},\quad
t>0,\quad  Y(0)=0
\]
(in the above definition it is meant that
$\sum_{n=1}^{0}\overset{\mathrm{def}}=0$). This process can be also
represented as $Y(t)=G ( N(t) ) $, that is, as a
time-changed Gamma process, where the role of time is played by
Poisson process. Let us denote this process by $G_{N}(t)$:
\[
G_{N}(t)=G \bigl( N(t) \bigr) .
\]
Let $N(t)$ and $G(t)$ have parameters $\lambda$ and $(\alpha
,\beta)$ correspondingly. The process $G_{N}(t)$\ is a L\'{e}vy
process with Laplace exponent (or Bern\v{s}tein function) of the
form
\[
f(u)=\lambda\beta^{\alpha} \bigl( \beta^{-\alpha}-(\beta
+u)^{-\alpha
}{} \bigr) ,\quad\lambda>0,\ \alpha>0,\ \beta>0,
\]
and the corresponding L\'{e}vy measure is
\begin{equation}
\nu(du)=\lambda\beta^{\alpha} \bigl( \varGamma(\alpha) \bigr)
^{-1}u^{\alpha-1}e^{-\beta u}du,\quad\lambda>0,\
\alpha >0,\ \beta>0. \label{measurePG}
\end{equation}

The transition probability measure of the process $G_{N}(t)$ can be written
in the closed form:
\begin{align}
P \bigl\{ G_{N}(t)\in ds \bigr\} &=e^{-\lambda t}\delta
_{\{0\}}(ds)+\sum_{n=1}^{\infty}e^{-\lambda t}
\frac{ ( \lambda
t\beta
^{\alpha} ) ^{n}}{n!\varGamma(\alpha n)}s^{\alpha n-1}e^{-\beta s}ds \label{probGN}
\\
&=e^{-\lambda t}\delta_{\{0\}}(ds)+e^{-\lambda t-\beta s}\frac
{1}{s}
\varPhi \bigl( \alpha,0,\lambda t(\beta s)^{\alpha} \bigr) ds,
\nonumber
\end{align}
therefore, probability law of $G_{N}(t)$ has atom $e^{-\lambda t}$
at zero, that is, has a discrete part $P \{ G_{N}(t)=0 \}
=e^{-\lambda t}$, and the density of the absolutely continuous
part can be expressed in terms of the Wright function.

In particular, when $\alpha=n$, $n\in N$, we have a
Poisson--Erlang process, which we will denote by $G_{N}^{(n)}(t)$
and for $\alpha=1$, we have a Poisson process with exponentially
distributed jumps. We will denote this last process by $E_{N}(t)$.
Its L\'{e}vy measure $\nu(du)=\lambda\beta e^{-\beta u}du$ and
Laplace exponent is
\[
f(u)=\lambda\frac{u}{\beta+u}{}.
\]

The transition probability measure of $E_{N}(t)$ is given by
\begin{align*}
P \bigl\{ E_{N}(t)\in ds \bigr\} &=e^{-\lambda t}\delta
_{\{0\}}(ds)+e^{-\lambda t-\beta s}\frac{1}{s}\varPhi ( 1,0,\lambda t
\beta s ) ds
\\
&=e^{-\lambda t}\delta_{\{0\}}(ds)+e^{-\lambda t-\beta s}
\frac{\sqrt{
\lambda t\beta s}}{s}I_{1} ( 2\sqrt{\lambda t\beta s} )ds .
\end{align*}

\begin{remark}\label{remark1}
Measures (\ref{measurePG}) belong to the class of L\'{e}vy
measures
\begin{equation}
\nu(du)=cu^{-a-1}e^{-bu}du,\quad c>0,\ a<1,\ b>0.
\label{measureGen}
\end{equation}
Note that:
\begin{enumerate}
\item[(i)] for the range $a\in(0,1)$ and $b>0$ we obtain the
tempered Poisson processes,

\item[(ii)] the limiting case when $a\in(0,1)$, $b=0$
corresponds to stable subordinators,

\item[(iii)] the case $a=0$, $b>0$ corresponds to Gamma
subordinators,

\item[(iv)] for $a<0$ we have compound Poisson-Gamma
subordinators.
\end{enumerate}

Probability distributions of the above subordinators can be
written in the closed form in the case (iii); in the case (i)
for $\alpha=1/2$, when we have the inverse Gaussian subordinators;
and in the case (iv).

In the paper \cite{OT} the deep and detailed investigation was
performed for the time-changed Poisson processes where the role of
time is played by the subordinators from the above cases
(i)--(iii) (and as we have already pointed out in the introduction, the
most studied in the literature is the case (ii)). The mentioned
paper deals actually with the general
time-changed processes of the form $N(H^{f}(t))$, where $%
N(t)$ is a Poisson process and $H^{f}(t)$ is an arbitrary
subordinator with Laplace exponent $f$, independent of $N(t)$.
This general construction falls into the framework of Bochner
subordination (see, e.g. book by Sato \cite{Sato}) and can be
studied by means of different approaches.

With the present paper we intend to complement the study
undertaken in \cite{OT} with one more example. We consider
time-change by means of subordinators corresponding to the above
case (iv) and, therefore, subordination related to the measures
\eqref{measureGen} will be covered for the all range of
parameters. We must also admit that the attractive feature of
these processes is the closed form of their distributions which
allows to perform the exact calculations for characteristics of
corresponding time-changed processes.

We also study Skellam processes with time-change by means of
compound Pois\-son-Gamma subordinators, in this part our results are
close to the corresponding results of the paper \cite{KLS}.

In our paper we develop (following \cite{OT,KLS} among
others) an approach for studying time-changed Poisson process via
investigation of their distributional properties, with the use of
the form and properties of distributions of the processes
involved.

It would be also interesting to study the above mentioned
processes within the framework of Bochner subordination via
semigroup approach. We address this topic for future research, as
well as the study of other characteristics of the processes
considered in the present paper and their comparison with related
results existing in the literature.

Note that in our exposition for convenience we will write L\'{e}vy
measures \eqref{measureGen} for the case of compound Poisson-Gamma
subordinators in the form (\ref{measurePG}) with $\alpha>0$ and
transparent meaning of parameters.
\end{remark}

\begin{remark}\label{remark2}
It is well known that the composition of two independent stable
subordinators is again a stable subordinator. If $S_{i}(t)$,
$i=1,2$, are two independent
stable subordinators with Laplace exponents $c_{i}u^{a_{i}}$, $i=1,2$,
then $%
S_{1}(S_{2}(t))$ is the stable\vadjust{\goodbreak} subordinator with index $a_{1}a_{2}$, since
its Laplace exponent is $c_{2} ( c_{1} ) ^{a_{2}}u^{a_{1}a_{2}}$.
More generally, iterated composition of stable subordinators $%
S_{1},\ldots,S_{k} $ with indices $a_{1},\ldots,a_{k}$ is the stable subordinator
with index $\prod_{i=1}^{k}a_{i}$.

In the paper \cite{KS} it was shown that one specific subclass of
Poisson-Gamma subordinators has a property similar to the property of stable
subordinators described above, that is, if two processes belong to this
class, so does their composition. Namely, this is the class of compound
Poisson processes with exponentially distributed jumps and parameters $
\lambda=\beta=\frac{1}{a}$, therefore, the L\'{e}vy measure is of
the form
$\nu(du)=\frac{1}{a^{2}}e^{-u/a}du$ and the corresponding Bern\v
{s}tein function is
$\tilde{f}_{a}(u)=\frac{u}{1+au}$. Denote such processes by $\widetilde
{E}%
_{N}^{a}(t)$. Then, as can be easily checked (see also \cite{KS}) the
composition of two independent processes $\widetilde{E}_{N}^{a_{1}}(%
\widetilde{E}_{N}^{a_{2}}(t))$ has Laplace exponent $\tilde{f}%
_{a_{1}+a_{2}}(u)=\frac{u}{1+(a_{1}+a_{2})u}$, since the functions
\[
\tilde{f}_{a}(u)=\frac{u}{1+au}
\]
have the property
\[
\tilde{f}_{a}\bigl(\tilde{f}_{b}(\lambda)\bigr)=
\tilde{f}_{a+b}(\lambda).
\]
Therefore, the process $\widetilde{E}_{N}^{a_{1}}(\widetilde{E}%
_{N}^{a_{2}}(t))$ coincides, in distribution, with $\widetilde
{E}_{N}^{a}(t)$, with a parameter $a=a_{1}+a_{2}$. More generally, iterated
composition of $%
n$ independent processes
\[
\widetilde{E}_{N}^{a_{1}}\bigl(\widetilde{E}_{N}^{a_{2}}
\bigl(\ldots\widetilde{E}%
_{N}^{a_{n}}(t)\ldots\bigr)\bigr)
\]
is equal in distribution to the process $\widetilde{E}_{N}^{a}(t)$ with
$%
a=\sum_{i=1}^{n}a_{i}$, and the $n$-th iterated composition as
$n\rightarrow
\infty$ results in the process $\widetilde{E}_{N}^{a}(t)$ with $%
a=\sum_{i=1}^{\infty}a_{i}$, provided that $0<a<\infty$. In particular,
subordinator $\widetilde{E}_{N}^{1}(t)$ with Laplace exponent
$f(u)=\frac{u}{%
1+u}$ can be represented as infinite composition of independent
subordinators $\widetilde{E}_{N}^{1/2^{n}}$:
\[
\widetilde{E}_{N}^{1}(t)\stackrel{d} {=}\widetilde
{E}_{N}^{1/2}\bigl(\widetilde{E}%
_{N}^{1/2^{2}}
\bigl(\ldots\widetilde{E}_{N}^{1/2^{n}}(\ldots)\bigr)\bigr).
\]
This interesting feature of the processes
$\widetilde{E}_{N}^{a}(t)$ deserves further investigation.
\end{remark}

\section{Compound Poisson-Gamma process as time change}
\label{sec:com}

Let $N_{1}(t)$ be the Poisson process with intensity $\lambda_{1}$.
Consider the time-changed process $N_{1}(G_{N}(t))=N_{1}(G(N(t)))$,
where $%
G_{N}(t)$ is compound Poisson-Gamma process (with parameters $\lambda
,\alpha,\beta$) independent of $N_{1}(t)$.

\begin{theorem}\label{theorem1}
Probability mass function of the process $X(t)=N_{1}(G_{N}(t))$ is
given by
\begin{equation}
p_{k}(t)=P\bigl(X(t)=k\bigr)=\frac{e^{-\lambda t}}{k!}\frac{\lambda
_{1}^{k}}{(\lambda_{1}+\beta)^{k}}\sum
_{n=1}^{\infty
}\frac{ ( \lambda t\beta^{\alpha} ) ^{n}\varGamma(\alpha
n+k)}{(\lambda_{1}+\beta)^{\alpha n}n!\varGamma(\alpha n)} \quad\mbox{
\text{for}}\  k\ge1 \label{pkNGN},
\end{equation}
and
\begin{equation}
p_{0}(t)=\exp \biggl\{-\lambda t \biggl(1-\frac{\beta^{\alpha
}}{(\lambda_{1}+\beta)^{\alpha}} \biggr)
\biggr\}\label{p0NGN}.
\end{equation}
The probabilities $p_{k}(t)$, $k\geq0$, satisfy the following
system of difference-differential equations:\vadjust{\goodbreak}
\[
\frac{d}{dt}p_{k}(t)= \biggl( \frac{\lambda\beta^{\alpha}}{(\lambda
_{1}+\beta)^{\alpha}}-\lambda
\biggr) p_{k}(t)+\frac{\lambda\beta
^{\alpha}}{(\lambda_{1}+\beta)^{\alpha}}\sum_{m=1}^{k}
\frac{\lambda
_{1}^{m}}{ ( \lambda_{1}+\beta ) ^{m}}\frac{\varGamma
(m+\alpha)}{%
m!\varGamma(\alpha)}p_{k-m}(t).
\]
\end{theorem}

 In the case when $a=1$, that is, when the process $G_{N}(t)$
becomes $E_{N}(t)$, the compound Poisson process with
exponentially distributed jumps, probabilities $p_{k}(t)$, $k\geq
0$, can be represented via the generalized Mittag-Leffler function
as stated in the next theorem.

\begin{theorem}\label{theorem2}
Let $X(t)=N_{1}(E_{N}(t))$. Then for $k\ge1$
\begin{align}
p_{k}^{E}(t)&=P\bigl(X(t)=k\bigr)=e^{-\lambda t}
\frac{\lambda_{1}^{k}\lambda
\beta t}{%
(\lambda_{1}+\beta)^{k+1}}\mathcal{E}_{1,2}^{k+1} \biggl(
\frac{\lambda\beta t}{\lambda_{1}+\beta} \biggr) \label{pkNEN};\\
p_{0}(t)&=\exp \biggl\{- \frac{\lambda\lambda_1 t }{\lambda
_{1}+\beta} \biggr\}\label{p0NEN},
\end{align}
and the probabilities $p_{k}(t)$, $k\ge0$, satisfy the following
equation:
\begin{equation}
\frac{d}{dt}p_{k}^{E}(t)=-\lambda
\frac{\lambda_{1}}{\lambda
_{1}+\beta}%
p_{k}^{E}(t)+\frac{\lambda\beta}{\lambda_{1}+\beta}
\sum_{m=1}^{k} \biggl( \frac{\lambda_{1}}{\lambda_{1}+\beta}
\biggr) ^{m}p_{k-m}^{E}(t). \label{dpkNEN}
\end{equation}
\end{theorem}

\begin{proof}[Proof of Theorem \ref{theorem1}] The probability mass
function of the process $X(t)=\break N_{1}(G_{N}(t))$ can be obtained by
standard conditioning arguments (see, e.g., the general result for
subordinated L\'{e}vy processes in \cite{Sato}, Theorem 30.1).

For $k\ge1$ we obtain:
\begin{align*}
p_{k}(t) &=P\bigl(N_{1}\bigl(G_{N}(t)\bigr)=k
\bigr)=\int_{0}^{\infty
}P\bigl(N_{1}(s)=k
\bigr)P\bigl(G_{N}(t)\in ds\bigr)
\\
&=\int_{0}^{\infty}e^{-\lambda_{1}s}\frac{ ( \lambda_{1}s
) ^{k}%
}{k!}
\biggl\{ e^{-\lambda t}\delta_{\{0\}}(ds)+e^{-\lambda t-\beta
s}
\frac{1%
}{s}\varPhi \bigl( \alpha,0,\lambda t(\beta s)^{\alpha} \bigr)
\biggr\} ds
\\
&=\frac{e^{-\lambda t}\lambda_{1}^{k}}{k!}\int_{0}^{\infty
}e^{-\lambda
_{1}s}s^{k}e^{-\beta s}
\frac{1}{s}\varPhi \bigl( \alpha,0,\lambda t(\beta s)^{\alpha} \bigr)
ds
\\
&=\frac{e^{-\lambda t}\lambda_{1}^{k}}{k!}\sum_{n=1}^{\infty}
\frac
{ (
\lambda t\beta^{\alpha} ) ^{n}}{n!\varGamma(\alpha n)}\int_{0}^{\infty
}e^{-(\lambda_{1}+\beta)s}s^{k+\alpha n-1}ds
\\
&=\frac{e^{-\lambda t}}{k!}\frac{\lambda_{1}^{k}}{(\lambda
_{1}+\beta)^{k}%
}\sum_{n=1}^{\infty}
\frac{ ( \lambda t\beta^{\alpha} )
^{n}\varGamma(\alpha n+k)}{(\lambda_{1}+\beta)^{\alpha n}n!\varGamma
(\alpha n)}%
.
\end{align*}

For $k=0$ we have:
\begin{align*}
p_{0}(t)&= e^{-\lambda t}+e^{-\lambda t}\sum
_{n=1}^{\infty} \frac{ (\lambda t \beta^{\alpha} )^n \varGamma(\alpha n)}{
(\lambda_1+\beta )^{\alpha n} n!\varGamma(\alpha n)}
\\
&=e^{-\lambda t}\sum_{n=0}^{\infty}
\frac{ (\lambda t
\beta^{\alpha} )^n}{ (\lambda_1+\beta )^{\alpha n}
n!}=e^{-\lambda t} e^{\frac{\lambda t
\beta^{\alpha}}{ (\lambda_1+\beta )^{\alpha}}}
\\
&= \exp \biggl\{-\lambda t \biggl(1-\frac{\beta^{\alpha}}{(\lambda
_{1}+\beta)^{\alpha}} \biggr) \biggr\}.
\end{align*}

The governing equation for the probabilities
$p_{k}(t)=P(N_{1}(G_{N}(t))=k)$ follows as a particular case of
the general equation presented in Theorem 2.1 \cite{OT} for
probabilities $P(N(H(t))=k)$, where $H(t)$ is an arbitrary
subordinator.
\end{proof}

\begin{proof}[Proof of Theorem \ref{theorem2}] 
We obtain statements of Theorem \ref{theorem2} as
consequences of corresponding statements of Theorem \ref{theorem1}, using the
expansion (%
\ref{ML3}) for three-parameter generalized Mittag-Leffler function.
\end{proof}

\begin{remark}\label{remark3}
Moments of the process $N_{1}(G_{N}(t))$ can be calculated, for
example, from the moment generating function which is given by:
\begin{equation}
\mathsf{E}e^{\theta N_{1}(G_{N}(t))}=e^{-t\lambda ( 1-\beta
^{\alpha
}(\beta+\lambda_{1}(1-e^{\theta}))^{-\alpha} ) },\text{ \ } \label{mgNGN}
\end{equation}
for $\theta\in\Bbb{R}$ such that $\beta+\lambda_{1}(1-e^{\theta
})\neq
0$. We have, in particular,
\begin{equation*}
\mathsf{E}N_{1}\bigl(G_{N}(t)\bigr)=\lambda_{1}
\lambda\alpha\beta ^{-1}t,\qquad \mathsf{Var}N_{1}\bigl(G_{N}(t)
\bigr)=\lambda_{1}\lambda\alpha\beta ^{-2}\bigl((1+\alpha)
\lambda_{1}+\beta\bigr)t.
\end{equation*}
Expressions for probabilities $p_{k}(t)$ and expressions for
moments were also calculated in \cite{KS} using the probability
generating function of the process $N_{1}(G_{N}(t))$. In \cite{KS}
the covariance function was also obtained:
\begin{equation}
\mathsf{Cov}\bigl(N_{1}\bigl(G_{N}(t)\bigr),N_{1}
\bigl(G_{N}(s)\bigr)\bigr)=\lambda_{1}\lambda\alpha
\beta^{-2}\bigl((1+\alpha)\lambda_{1}+\beta\bigr)\min
(t,s).\label{covGN}
\end{equation}

The very detailed study of time-changed processes $N(H^{f}(t))$ with $%
H^{f}(t)$ being an arbitrary subordinator, independent of Poisson
process $%
N(t)$, with Laplace exponent $f(u)$, is presented in \cite{OT}. In
particular, it was shown therein that the probability generating function
can be written in the form $G(u,t)=e^{-tf(\lambda(1-u))}$, where
$\lambda$
is the parameter of the process $N(t)$.

Note also that in order to compute the first two moments and
covariance function of time-changed L\'{e}vy processes the
following result, stated in \cite{LMSS} as Theorem 2.1, can be
used:

If $X(t)$ is a homogeneous L\'{e}vy process with $X(0)=0$, $Y(t)$ is a
nondecreasing process independent of $X(t)$ and $Z(t)=X(Y(t))$, then
\begin{equation}
\mathsf{E}Z(t)=\mathsf{E}X(1)U(t), \label{expComp}
\end{equation}
provided that $\mathsf{E}X(t)$ and $U(t)=\mathsf{E}Y(t)$ exist;

and if $X(t)$ and $Y(t)$ have finite second moments, then
\begin{align}
\mathsf{E} \bigl[ Z(t) \bigr] ^{2}&=\mathsf{Var}X(1)U(t)- \bigl[
\mathsf {E}X(1)%
 \bigr] ^{2}\mathsf{E} \bigl[ Y(t) \bigr]
^{2}, \label{mom2Comp}\\
\mathsf{Var}Z(t)&= \bigl[ \mathsf{E}X(1) \bigr] ^{2}\mathsf {Var}Y(t)+
\mathsf{%
Var}X(1)U(t), \label{VarComp}\\
\mathsf{Cov}\bigl(Z(t),Z(s)\bigr)&= \bigl[ \mathsf{E}X(1) \bigr] ^{2}
\mathsf{Cov}%
\bigl(Y(t),Y(s)\bigr)+\mathsf{Var}X(1)U\bigl(\min(t,s)
\bigr). \label{CovComp}
\end{align}
We will use the above formulas further in the paper.
\end{remark}

\begin{remark}\label{remark4}
Let $\widetilde{E}^{a}(t)=\widetilde{E}_{N}^{a}(t)$ be the
compound
Poisson-exponential process with Laplace exponent $\tilde
{f}_{a}(u)=\frac{u}{%
1+au}$, which we discussed in Remark \ref{remark2} above, and let $N(t)=N_{\lambda}(t)$
be the Poisson process with parameter $\lambda$, independent of $%
\widetilde{E}^{a}(t)$. Denote
$\widetilde{N}^{a}(t)=N(\widetilde{E}^{a}(t))$. In view of Remark
\ref{remark2}, for the processes $\widetilde{N}^{a}(t)$ we have the following
property concerning double and iterated compositions:
\[
\widetilde{N}^{a_{1}}\bigl(\widetilde{E}^{a_{2}}(t)\bigr)={N}
\bigl(\widetilde {E}^{a_{1}}\bigl(\widetilde{E}^{a_{2}}(t)\bigr)
\bigr)\stackrel{d} {=}\widetilde{N}%
^{a_{1}+a_{2}}(t)
\]
and
\[
\widetilde{N}^{a_{1}}\bigl(\widetilde{E}_{N}^{a_{2}}
\bigl(\ldots\widetilde{E}%
_{N}^{a_{n}}(t)\ldots\bigr)\bigr)
\stackrel{d} {=}\widetilde{N}^{\sum_{i=1}^{n}a_{i}}(t).
\]
This property makes processes $\widetilde{N}^{a}(t)$ similar to the
space-fractional Poisson processes, which are obtained as
\[
{N}^{\alpha}(t)=N\bigl({S}^{\alpha}(t)\bigr)
\]
where ${S}^{\alpha}(t)$ is a stable subordinator with $\mathsf
{E}e^{-\theta
{S}^{\alpha}(t)}=e^{-t\theta^{\alpha}}$. In the papers \cite
{OP,GOS} it was shown that
\[
{N}^{\alpha}\bigl({S}^{\gamma}(t)\bigr)\stackrel{d}
{=}N^{\alpha\gamma}(t)
\]
and
\begin{equation}
{N}^{\alpha}\bigl(S{^{\gamma_{1}}}\bigl(\ldots S{^{\gamma_{n}}}(t)\ldots\bigr)
\bigr)\stackrel {d} {=} {N%
}^{\alpha\prod_{i=1}^{n}\gamma_{i}}(t), \label{conserv}
\end{equation}
relation (\ref{conserv}) is referred to in \cite{GOS} as
auto-conservative property. This property deserves the further
investigation.
\end{remark}

\begin{remark}\label{remark5}
Note that the marginal laws of the time-changed processes obtained in
Theorems \ref{theorem1}, \ref{theorem2} (and in the next theorems in what follows) can be considered
as new classes of distributions, in particular, (\ref{pkNEN}) and (\ref
{pkNY}%
) below represent three-parameter distributions involving the generalized
Mittag-Leffler functions.
\end{remark}

We now consider Skellam processes $S(t)$ with time change, where the
role of
time is played by compound Poisson-Gamma subordinators $G_{N}(t)$ with
Laplace exponent $f(u)=\lambda\beta^{\alpha} ( \beta^{-\alpha
}-(\beta+u)^{-\alpha}{} ) $, $\lambda>0$, $\alpha>0$, $\beta>0$.

Let the process $S(t)$ have parameters $\lambda_{1}$ and $\lambda
_{2}$ and let us
consider first the time-changed Skellam process of type I, that is the
process
\[
S_{_{I}}(t)=S\bigl(G_{N}(t)\bigr)=N_{1}
\bigl(G_{N}(t)\bigr)-N_{2}\bigl(G_{N}(t)\bigr),
\]
where $N_{1}(t)$, $N_{2}(t)$ and $G_{N}(t)$ are mutually
independent.

\begin{theorem}\label{theorem3}
Let $S_{_{I}}(t)=S(G_{N}(t))$, then probabilities
$r_{k}(t)=P(S_{_{I}}(t)=k)$ are given by
\begin{equation}
r_{k}(t)=e^{-\lambda t} \biggl( \frac{\lambda_{1}}{\lambda_{2}} \biggr)
^{%
\frac{k}{2}}\int_{0}^{\infty}e^{-u ( \lambda_{1}+\lambda
_{2}+\beta
 ) }I_{|k|}
( 2u\sqrt{\lambda_{1}\lambda_{2}} ) \frac
{1}{u}%
\varPhi \bigl( \alpha,0,\mu t ( \beta u ) ^{\alpha} \bigr) du. \label{SIGN}
\end{equation}
The moment generating function of $S_{_{I}}(t)$ has the following
form:
\[
\mathsf{E}e^{\theta S_{I}(t)}=e^{-\lambda t ( 1-\beta^{\alpha
} (
\beta+\lambda_{1}+\lambda_{2}-\lambda_{1}e^{\theta}-\lambda
_{2}e^{-\theta} ) ^{-\alpha} ) }
\]
for $\theta$ such that $\beta+\lambda_{1}+\lambda_{2}-\lambda
_{1}e^{\theta}-\lambda_{2}e^{-\theta}\neq0$.
\end{theorem}

\begin{remark}\label{remark6}
For the case $\alpha=1$, that is, $S_{_{I}}(t)=S(E_{N}(t))$, we
obtain
\begin{align*}
r_{k}(t)&=e^{-\lambda t} \biggl( \frac{\lambda_{1}}{\lambda_{2}} \biggr)
^{%
\frac{k}{2}}\int_{0}^{\infty}e^{-u ( \lambda_{1}+\lambda
_{2}+\beta
 ) }I_{|k|}
( 2u\sqrt{\lambda_{1}\lambda_{2}} ) \sqrt {
\frac{%
\lambda\beta t}{u}}I_{1} ( 2\sqrt{\lambda t\beta u} ) du,\\
\mathsf{E}e^{\theta S_{I}(t)}&=e^{-\lambda t ( \lambda_{1}+\lambda
_{2}-\lambda_{1}e^{\theta}-\lambda_{2}e^{-\theta} )  (
\beta
+\lambda_{1}+\lambda_{2}-\lambda_{1}e^{\theta}-\lambda
_{2}e^{-\theta
} ) ^{-1}}.
\end{align*}
\end{remark}

\begin{proof}[Proof of Theorem \ref{theorem3}] Using conditioning arguments, we can
write:
\[
r_{k}(t)=\int_{0}^{\infty}s_{k}(u)P
\bigl(G_{N}(t)\in du\bigr),
\]
where $s_{k}(u)=P(S(t)=k)$. Inserting the expressions for $s_{k}(u)$
and $%
P(G_{N}(t)\in du)$, which are given by formulas (\ref{skellamdistr})
and (%
\ref{probGN}) correspondingly, we come to (\ref{SIGN}). The moment
generating function can be obtained as follows:
\begin{align*}
\mathsf{E}e^{\theta S_{I}(t)} &=\int_{0}^{\infty}
\mathsf{E}e^{\theta
S(t)}P\bigl(G_{N}(t)\in du\bigr)=
\mathsf{E}e^{-f_{s}(-\theta)G_{N}(t)}
\\
&=\mathsf{E}e^{- ( \lambda_{1}+\lambda_{2}-\lambda_{1}e^{\theta
}-\lambda_{2}e^{-\theta} ) G_{N}(t)}
\\
&=e^{-t\lambda ( 1-\beta^{\alpha} ( \beta+\lambda
_{1}+\lambda
_{2}-\lambda_{1}e^{\theta}-\lambda_{2}e^{-\theta} ) ^{-\alpha
} ) }
\end{align*}
for $\theta$ such that $\beta+\lambda_{1}+\lambda_{2}-\lambda
_{1}e^{\theta}-\lambda_{2}e^{-\theta}\neq0$.
\end{proof}

\begin{remark}\label{remark7}
The mean, variance and covariance function for Skellam process
$S_I(t)=S(G_N(t))$ can be calculated
with the use of the general result for time-changed L\'{e}vy
processes stated in \cite{LMSS}, Theorem 2.1 (see our Remark \ref{remark3},
formulas (\ref{expComp})--(\ref{CovComp})) and expressions for the
mean, variance and covariance of Skellam process, which are
\begin{align*}
\mathsf{E}S(t) &= (\lambda_1-\lambda_2 ) t;\\
\mathsf{Var}S(t) &= (\lambda_1+\lambda_2 )t;\\
\mathsf{Cov}\bigl(S(t),S(s)\bigr) &= (\lambda_1+
\lambda_2 )\min(t,s).
\end{align*}
We obtain:
\begin{align*}
\mathsf{E}S\bigl(G_N(t)\bigr) &= (\lambda_1-
\lambda_2 ) \alpha\beta ^{-1} \lambda t;\\
\mathsf{Var}S\bigl(G_N(t)\bigr)& =\alpha\beta^{-2}
\lambda \bigl[ \bigl(\lambda _1-\lambda_2(\alpha+1)
\bigr)^2+ (\lambda_1+\lambda_2 )\beta
\bigr]t;\\
\mathsf{Cov}\bigl(S\bigl(G_N(t)\bigr),S\bigl(G_N(s)
\bigr)\bigr) &= \alpha\beta^{-2} \lambda \bigl[ \bigl(
\lambda_1-\lambda_2(\alpha+1) \bigr)^2+ (
\lambda_1+\lambda _2 )\beta \bigr]\min(t,s).
\end{align*}
\end{remark}

Consider now the time-changed Skellam process of type II, where the
role of
time is played by the subordinator $X(t)=E_{N}(t)$ with Laplace
exponent $%
f(u)=\lambda\frac{u}{\beta+u}$, that is, the process
\begin{equation}
S_{\mathit{II}}(t)=N_{1}\bigl(X_{1}(t)
\bigr)-N_{2}\bigl(X_{2}(t)\bigr), \label{s2exp}
\end{equation}
where $X_{1}(t)$ and $X_{2}(t)$ are independent copies $X(t)$ and
independent of $N_{1}(t)$, $N_{2}(t)$.

\begin{theorem}\label{theorem4}
Let $S_{\mathit{II}}(t)$ be the time-changed Skellam process of type II
given by (\ref{s2exp}). Its probability mass function is given by\vadjust{\goodbreak}
\begin{align}
P\bigl(S_{\mathit{II}}(t)=k\bigr) &=e^{-2\lambda t}\frac{ ( \lambda t\beta )
^{2}\lambda_{1}^{k}}{(\lambda_{1}+\beta)^{k+1}(\lambda_{2}+\beta)}
\sum_{n=0}^{\infty}
\frac{ ( \lambda_{1}\lambda_{2} ) ^{n}}{
(\lambda_{1}+\beta)^{n}(\lambda_{2}+\beta)^{n}}
\nonumber
\\
&\quad\times\,\mathcal{E}_{1,2}^{n+k+1} \biggl( \frac{\lambda\beta
t}{\lambda
_{1}+\beta}
\biggr) \mathcal{E}_{1,2}^{n+1} \biggl( \frac{\lambda\beta
t}{%
\lambda_{2}+\beta}
\biggr) \label{ps2_1}
\end{align}
for $k\in Z$, $k\geq0$, and when $k<0$
\begin{align}
P\bigl(S_{\mathit{II}}(t)=k\bigr) &=e^{-2\lambda t}\frac{ ( \lambda t\beta )
^{2}\lambda_{2}^{ \llvert  k \rrvert  }}{(\lambda_{1}+\beta)(\lambda
_{2}+\beta)^{ \llvert  k \rrvert  +1}}\sum
_{n=0}^{\infty}\frac{ (
\lambda
_{1}\lambda_{2} ) ^{n}}{(\lambda_{1}+\beta)^{n}(\lambda
_{2}+\beta
)^{n}}
\nonumber
\\
&\quad\times\,\mathcal{E}_{1,2}^{n+1} \biggl( \frac{\lambda\beta t}{\lambda
_{1}+\beta}
\biggr) \mathcal{E}_{1,2}^{n+ \llvert  k \rrvert  +1} \biggl( \frac{%
\lambda\beta t}{\lambda_{2}+\beta}
\biggr) . \label{ps2_2}
\end{align}
The moment generating function is
\[
\mathsf{E}e^{\theta S_{\mathit{II}}(t)}=e^{-t\lambda\lambda_{1}(1-e^{\theta
})\,/\,(\beta+\lambda_{1}(1-e^{\theta}))}e^{-t\lambda\lambda
_{2}(1-e^{-\theta})\,/\,(\beta+\lambda_{2}(1-e^{-\theta}))}
\]
for $\theta$ such that $\beta+\lambda_{1}(1-e^{\theta})\neq
0$.
\end{theorem}

\begin{proof}[Proof of Theorem \ref{theorem4}] Using the independence of $%
N_{1}(X_{1}(t))$ and $N_{2}(X_{2}(t))$, we can write:
\begin{align}
P \bigl( S_{\mathit{II}}(t)=k \bigr) &=\sum_{n=0}^{\infty}P
\bigl\{ N_{1}\bigl(X_{1}(t)\bigr)=n+k \bigr\} P \bigl\{
N_{2}\bigl(X_{2}(t)\bigr)=n \bigr\} I_{
\{
k\geq0 \} }
\nonumber\\
&\quad+\sum_{n=0}^{\infty}P \bigl\{
N_{1}\bigl(X_{1}(t)\bigr)=n \bigr\} P \bigl\{
N_{2}\bigl(X_{2}(t)\bigr)=n+|k| \bigr\} I_{ \{ k<0 \} }
\label{proofSII}
\end{align}
and then we use the expressions for probabilities of $N_{i}(E_{N}(t))$ given
in Theorem \ref{theorem2}. For $k>0$ we have the expression
\begin{align*}
P \bigl( S_{\mathit{II}}(t)=k \bigr) &=\sum_{n=0}^{\infty}e^{-2\lambda t}
\frac{
\lambda_{1}^{n+k}\lambda t\beta}{(\lambda_{1}+\beta
)^{n+k+1}}\mathcal{E}%
_{1,2}^{n+k+1} \biggl(
\frac{\lambda\beta t}{\lambda_{1}+\beta} \biggr)
\\
&\quad\times\frac{\lambda_{2}^{n}\lambda t\beta}{(\lambda_{2}+\beta
)^{n+1}}%
\mathcal{E}_{1,2}^{n+1}
\biggl( \frac{\lambda\beta t}{\lambda2+\beta}%
 \biggr) ,
\end{align*}
from which we obtain (\ref{ps2_1}); and in the analogous way we come to
(\ref
{ps2_2}).  In view of independence of $N_{1}(X_{1}(t))$ and $%
N_{2}(X_{2}(t))$, the moment generating function is obtained as the product:
\[
\mathsf{E}e^{\theta S_{\mathit{II}}(t)}=\mathsf{E}e^{\theta
N_{1}(X_{1}(t))}\mathsf{E}%
e^{-\theta N_{2}(X_{2}(t))},
\]
and then we use expression (\ref{mgNGN}) for $\alpha=1$.
\end{proof}

\begin{remark}\label{remark8}
The moments of $S_{\mathit{II}}(t)$ can be calculated using the moment
generating function given in Theorem \ref{theorem4}, or using the independence of
processes $N_i(X_i(t)), i=1,2$,
and corresponding expressions for the moments of
$N_i(X_i(t))$, $i=1,2$.
Since $N_1(X_1(t))$ and $N_2(X_2(t))$ are independent we can also
easily obtain the covariance function as follows:
\begin{align*}
&\mathsf{Cov} \bigl(S_{\mathit{II}}(t), S_{\mathit{II}}(s) \bigr)\\
&\quad= \mathsf{Cov}
\bigl(N_1\bigl(X_1(t)\bigr),N_1
\bigl(X_1(s)\bigr) \bigr)
+\mathsf{Cov} \bigl(N_2\bigl(X_2(t)\bigr),N_2
\bigl(X_2(s)\bigr) \bigr)\\
&\quad=\lambda\beta^{-2} \bigl[\beta(\lambda_1+
\lambda_2)+2\bigl(\lambda _1^2+
\lambda_2^2\bigr) \bigr]\min(t,s),
\end{align*}
where the expressions for covariance function of the
process $N_1(E_N(t))$ are used (see formula
\eqref{covGN} with $\alpha=1$).
\end{remark}

\section{Inverse compound Poisson-Gamma process as time change}\label{section5}

\subsection{{Inverse compound Poisson-exponential process}}\label{section5.1}

We first consider the process $E_{N}(t)$, the compound Poisson process with
exponentially distributed jumps with Laplace exponent $f(u)=\lambda
\frac{u}{%
\beta+u}$.

Define the inverse process (or first passage time):
\begin{equation}
Y(t)=\inf\bigl\{u\geq0;\text{ }E_{N}(u)>t\bigr\},\quad t\geq0.
\label{Y1}
\end{equation}
It is known (see, e.g., \cite{Cox}) that the process $Y(t)$ has
density
\begin{align}
h(s,t) &=\lambda e^{-\lambda s-\beta t}I_{0} ( 2\sqrt{\lambda \beta
st}%
 )
\nonumber
\\
&=\lambda e^{-\lambda s-\beta t}\varPhi ( 1,1,\lambda\beta st ) , \label{Ydensity}
\end{align}
and its Laplace transform is
\begin{equation}
\mathsf{E}e^{-\theta Y(t)}=\frac{\lambda}{\theta+\lambda}e^{-\beta
\frac{%
\theta}{\theta+\lambda}t}, \label{LaplaceForInverse}
\end{equation}
which can be also 
verified as follows:
\begin{align*}
\mathsf{E}e^{-\theta Y(t)} &=\int_{0}^{\infty}e^{-\theta u}
\lambda e^{-\lambda u-\beta t}I_{0} ( 2\sqrt{\lambda\beta ut} ) du
\\
&=\int_{0}^{\infty}e^{-\theta u}\lambda
e^{-\lambda u-\beta
t}\sum_{n=0}^{\infty}
\frac{ ( \lambda\beta ut ) ^{n}}{ (
n! ) ^{2}}du
\\
&=\lambda e^{-\beta t}\sum_{n=0}^{\infty}
\frac{ ( \lambda\beta
t ) ^{n}}{ ( n! ) ^{2}}\int_{0}^{\infty}e^{-(\theta
+\lambda
)u}u^{n}du
\\
&=\lambda e^{-\beta t}\sum_{n=0}^{\infty}
\frac{ ( \lambda\beta
t ) ^{n}}{n! ( \theta+\lambda ) ^{n+1}}=\frac{\lambda
}{%
\theta+\lambda}e^{-\beta\frac{\theta}{\theta+\lambda}t}.
\end{align*}
Moments of $Y(t)$ can be easily found by the direct calculations. We have
\begin{equation}
\mathsf{E}Y(t)\,{=}\,\frac{1}{\lambda} ( \beta t+1 ) ,\qquad\mathsf{E}%
Y(t)^{2}\,{=}\,
\frac{1}{\lambda^{2}} \bigl( \beta^{2}t^{2}+4\beta t+2 \bigr) ,
\qquad%
\mathsf{Var}Y(t)\,{=}\,\frac{1}{\lambda^{2}} ( 2\beta t+1 ) . \label{momentsY}
\end{equation}
For example, the first moment can be obtained as follows:
\begin{align*}
\mathsf{E}Y(t&)=\int_{0}^{\infty}uh(u,t)du=\int
_{0}^{\infty}u\lambda e^{-\lambda u-\beta t}I_{0}
( 2\sqrt{\lambda\beta ut} ) du
\\
&=\int_{0}^{\infty}u\lambda e^{-\lambda u-\beta t}\sum
_{k=0}^{\infty
}%
\frac{ ( \lambda\beta ut ) ^{k}}{ ( k! )
^{2}}du=
\lambda e^{-\beta t}\sum_{k=0}^{\infty}
\frac{ ( \lambda\beta t )
^{k}}{%
 ( k! ) ^{2}}\int_{0}^{\infty}u^{k+1}e^{-\lambda u}du
\\
&=\lambda e^{-\beta t}\sum_{k=0}^{\infty}
\frac{ ( \lambda\beta
t ) ^{k}(k+1)!}{ ( k! ) ^{2}\lambda^{k+2}}=\frac
{1}{\lambda}%
e^{-\beta t}\sum
_{k=0}^{\infty}\frac{ ( \beta t )
^{k}(k+1)}{k!}=%
\frac{1}{\lambda} ( \beta t+1 ) .
\end{align*}

The covariance function of the process $Y(t)$\ can be calculated
using the results on moments of the inverse subordinators stated
in \cite{VT}.

\begin{lemma}\label{lemma1} Let $Y(t)$ be the inverse process given by
(\ref
{Y1}). Then
\begin{equation}
\mathsf{Cov}\bigl(Y(t),Y(s)\bigr)=\frac{1}{\lambda^{2}} \bigl( 2\beta\min (t,s)+1
\bigr) . \label{covY}
\end{equation}
\end{lemma}

The proof of Lemma \ref{lemma1} is given in Appendix \ref{appendixa.1}.

\begin{remark}\label{remark9}
It is known that, generally, the inverse subordinator is a process
with non-stationary, non-independent increments (see, e.g.,
\cite{VT}). We have not investigated here this question for the
process $Y(t)$, however we can observe the same similarity between
the expressions for variance and covariance of $Y(t)$ as that
which holds for the processes with stationary independent
increments.
\end{remark}

\begin{remark}\label{remark10}
Note that, for inverse processes, the important role is played by
the function $U(t)=\mathsf{E}Y(t)$, which is called the renewal
function. This function can be calculated using the following
formula for its Laplace transform:
\[
\int_{0}^{\infty}U(t)e^{-st}dt=
\frac{1}{sf(s)},
\]
where $f(s)$ is the Laplace exponent of the subordinator, for
which $Y(t)$ is the inverse (see, \cite{VT}, formula (13)).

In our case we obtain
\[
\int_{0}^{\infty}U(t)e^{-st}dt=
\frac{\beta+u}{\lambda u^{2}}=\frac{1}{
\lambda} \biggl( \frac{\beta}{u^{2}}+
\frac{1}{u^{{}}} \biggr) ,
\]
and by inverting this Laplace transform we come again to the
expression for the first moment given in (\ref{momentsY}). The
function $U(t)$ characterizes
the distribution of the inverse process $Y(t)$, since all the moments
of $%
Y(t)$ (and the mixed moments as well) can be expressed (by
recursion) in terms of $U(t)$ (see, \cite{VT}).
\end{remark}

Let $N_{1}(t)$ be the Poisson process with intensity $\lambda_{1}$.
Consider the time-changed process $Z(t)=N_{1}(Y(t))$, where $Y(t)$ is the
inverse process given by (\ref{Y1}), independent of $N_{1}(t)$.

\begin{theorem}\label{theorem5}
The probability mass function of the process $Z(t)=N_{1}(Y(t))$ is
given by
\begin{equation}
p_{k}^{I}(t)=P\bigl(Z(t)=k\bigr)=e^{-\beta t}
\frac{\lambda\lambda
_{1}^{k}(k+1)}{%
(\lambda_{1}+\lambda)^{k+1}}\mathcal{E}_{1,1}^{k+1} \biggl(
\frac
{\lambda
t\beta}{\lambda_{1}+\lambda} \biggr) , \label{pkNY}
\end{equation}
the mean and variance are:
\begin{equation}
\mathsf{E}Z(t)=\frac{\lambda_{1}}{\lambda} ( \beta t+1 ) ,\qquad \mathsf{Var}Z(t)=\frac{\lambda_{1}^{2}}{\lambda^{2}}
( 2\beta t+1 ) +%
\frac{\lambda_{1}}{\lambda} ( \beta t+1 ) ,
\end{equation}
and the covariance function has the following form:
\begin{equation}
\mathsf{Cov}\bigl(Z(t),Z(s)\bigr)=\frac{\lambda_{1}^{2}}{\lambda^{2}} \bigl( 2\beta \min(t,s)+1
\bigr) +\frac{\lambda_{1}}{\lambda} \bigl( \beta\min (t,s)+1 \bigr) .
\end{equation}

The Laplace transform is given by
\begin{equation}
\mathsf{E}e^{-\theta Z(t)}=\frac{\lambda}{\lambda+\lambda
_{1}(1-e^{-\theta})\,}e^{-\beta t\lambda_{1}(1-e^{-\theta}) (
\lambda
+\lambda_{1}(1-e^{-\theta}) ) ^{-1}}. \label{26}
\end{equation}
\end{theorem}

\begin{proof}[Proof of Theorem \ref{theorem5}] The probabilities
$p_{k}^{I}(t)=P(Z(t)=k)$ can be obtained by means of the following
calculations:
\begin{align*}
p_{k}^{I}(t) &=P\bigl(N_{1}\bigl(Y(t)\bigr)=k
\bigr)=\int_{0}^{\infty
}P\bigl(N_{1}(s)=k
\bigr)h(s,t)ds
\\
&=\int_{0}^{\infty}e^{-\lambda_{1}s}\frac{ ( \lambda_{1}s
) ^{k}%
}{k!}
\lambda e^{-\lambda s-\beta t}I_{0} ( 2\sqrt{\lambda\beta st} 
 )
ds
\\
&=\int_{0}^{\infty}e^{-\lambda_{1}s-\lambda s-\beta t}\lambda
\frac{
 ( \lambda_{1}s ) ^{k}}{k!}\sum_{n=0}^{\infty}
\frac{ (
\lambda\beta st ) ^{n}}{ ( n! ) ^{2}}ds
\\
&=\frac{e^{-\beta t}\lambda\lambda_{1}^{k}}{k!}\sum_{n=0}^{\infty
}
\frac{%
 ( \lambda\beta t ) ^{n}}{ ( n! ) ^{2}}\int_{0}^{\infty
}e^{-(\lambda_{1}+\lambda)s}s^{k+n}ds
\\
&=\frac{e^{-\beta t}\lambda\lambda_{1}^{k}}{k!}\sum_{n=0}^{\infty
}
\frac{%
 ( \lambda\beta t ) ^{n}\varGamma(n+k+1)}{ ( n! )
^{2}(\lambda_{1}+\lambda)^{n+k+1}}
\\
&=e^{-\beta t}\frac{\lambda\lambda_{1}^{k}(k+1)}{(\lambda
_{1}+\lambda
)^{k+1}}\mathcal{E}_{1,1}^{k+1}
\biggl( \frac{\lambda t\beta}{\lambda
_{1}+\lambda} \biggr) .
\end{align*}
The mean and variance can be calculated using formulas (\ref{expComp}),
(\ref
{VarComp}) and expressions (\ref{momentsY}).

The Laplace transform is obtained as follows:
\begin{align*}
\mathsf{E}e^{-\theta N_{1}(Y(t))} &=\int_{0}^{\infty}
\mathsf {E}e^{-\theta
N_{1}(t)}h(u,t)du=\mathsf{E}e^{-\lambda_{1} ( 1-e^{-\theta} )
Y(t)}
\\
&=\frac{\lambda}{\lambda+\lambda_{1}(1-e^{-\theta})\,}\exp \biggl( -t%
\frac{\beta\lambda_{1}(1-e^{-\theta})}{\lambda+\lambda
_{1}(1-e^{-\theta
})} \biggr) .\hspace*{40pt}\qedhere
\end{align*}
\end{proof}

We now state the relationship, in the form of a system of
differential equations, between the marginal distributions of the
processes $N_{1}(E_{N}(t))$ and $N_{1}(Y(t))$, with $Y(t)$ being
the inverse process for $E_{N}(t)$.\goodbreak

Introduce the following notations:
\begin{align}
p_{k}^{E}(t)&=p_{k}^{E}(t,\lambda,
\beta,\lambda _{1})=P\bigl(N_{1}\bigl(E_{N}(t)
\bigr)=k\bigr)=P\big(N_{1}^{\lambda_{1}}\xch{\bigl(E^{\beta}
\bigl(N^{\lambda}(t)\bigr)\big)}{\bigl(E^{\beta}
\bigl(N^{\lambda}(t)\bigr)}=k\bigr), \label{pke}\\
\hat{p}_{k}^{E}(t)&=p_{k}^{E}(t,
\beta,\lambda,\lambda_{1})=P\bigl(N_{1}\bigl(%
\widehat{E}_{N}(t)\bigr)=k\bigr)=P\xch{\big(N_{1}^{\lambda_{1}}
\bigl(E^{\lambda}\bigl(N^{\beta}(t)\bigr)\big)}{(N_{1}^{\lambda_{1}}
\bigl(E^{\lambda}\bigl(N^{\beta}(t)\bigr)}=k\bigr), \label{pkehat}
\end{align}
for $\hat{p}_{k}^{E}(t)$ we changed the order of parameters $\lambda$
and $%
\beta$, that is, parameters in $E$ and $N$ are interchanged, and
$\lambda$ and $%
\beta$ are now parameters of $E$ and $N$ correspondingly; and for the
inverse process we denote:
\begin{equation}
p_{k}^{I}(t)=p_{k}^{I}(t,\lambda,
\beta,\lambda_{1})=P\bigl(N_{1}\bigl(Y(t)\bigr)=k\bigr),
\label{pki}
\end{equation}
where $Y(t)$ is the inverse process for $E_{N}(t)=E^{\beta}\xch{(N^{\lambda}(t))}{(N^{\lambda}(t)}$,
\begin{equation}
\hat{p}_{k}^{I}(t)=p_{k}^{I}(t,
\beta,\lambda,\lambda_{1})=P\bigl(N_{1}\bigl(%
\widehat{Y}(t)\bigr)=k\bigr), \label{pkihat}
\end{equation}
where $\widehat{Y}(t)$ is the inverse process for $\widehat{E}%
_{N}(t)=E^{\lambda}\xch{(N^{\beta}(t))}{(N^{\beta}(t)}$.

\begin{theorem}\label{theorem6}
The probabilities $p_{k}^{E}(t)$, $\hat{p}_{k}^{E}(t)$ and
$p_{k}^{I}(t)$, $%
\hat{p}_{k}^{I}(t)$ satisfy the following differential equations:
\begin{equation}
\frac{d}{dt}\hat{p}_{k}^{E}(t)=-\beta
\hat{p}_{k}^{E}(t)+\frac{\beta
}{k+1}%
p_{k}^{I}(t)
\label{eq1}
\end{equation}
and
\begin{equation}
\frac{d}{dt}p_{k}^{E}(t)=-\lambda
p_{k}^{E}(t)+\frac{\lambda}{k+1}\hat {p}%
_{k}^{I}(t).
\label{eq2}
\end{equation}
If $\lambda=\beta$ then $p_{k}^{E}(t)$ and $p_{k}^{I}(t)$ satisfy the
following equation:
\begin{equation}
\frac{d}{dt}p_{k}^{E}(t)=-\lambda
p_{k}^{E}(t)+\frac{\lambda}{k+1}%
p_{k}^{I}(t).
\label{eq3}
\end{equation}
\end{theorem}

\begin{proof} Firstly we represent the derivative $\frac
{d}{dt}%
\hat{p}_{k}^{E}(t)$ in the form:
\begin{equation}
\frac{d}{dt}\hat{p}_{k}^{E}(t)=-\beta
\hat{p}_{k}^{E}(t)+e^{-\beta
t}\frac{%
\lambda_{1}^{k}}{(\lambda_{1}+\lambda)^{k}}
\frac{d}{dt}%
 \biggl( \frac{\lambda\beta t}{\lambda_{1}+\lambda}\mathcal{E}%
_{1,2}^{k+1}
\biggl( \frac{\lambda\beta t}{\lambda_{1}+\lambda} \biggr) \biggr) . \label{ddtpke}
\end{equation}
Next we notice that the following relation holds:
\[
\frac{d}{dz} \bigl( az\mathcal{E}_{1,2}^{\gamma} ( az )
\bigr) =a%
\mathcal{E}_{1,1}^{\gamma} ( az ) ,
\]
which can be easily checked by direct calculations.

Therefore, (\ref{ddtpke}) can be written as
\begin{equation}
\frac{d}{dt}\hat{p}_{k}^{E}(t)=-\beta
\hat{p}_{k}^{E}(t)+e^{-\beta
t}\frac{%
\lambda_{1}^{k}\lambda\beta}{(\lambda_{1}+\lambda)^{k+1}}
\mathcal {E}%
_{1,1}^{k+1} \biggl( \frac{\lambda\beta t}{\lambda_{1}+\lambda}
\biggr) . \label{ddtpke1}
\end{equation}
Comparing the second term in the r.h.s. of (\ref{ddtpke1}) with the
expression for $p_{k}^{I}(t)$, we come to (\ref{eq1}). With analogous
reasonings we derive (\ref{eq2}), and taking $\lambda=\beta$ in (\ref
{eq1}%
) or in (\ref{eq2}), we obtain (\ref{eq3}).
\end{proof}

\begin{remark}\label{remark11} We note that in addition to \eqref{eq1}--\eqref{eq3}
some other relations can be written,
in particular, the probabilities $p_{k}^{I}(t)$ can be
represented recursively via ${p}_{k}^{E}(t)$. If $\lambda=\beta
$, from \eqref{eq3} and \eqref{dpkNEN} we deduce:
\begin{equation*}
p_{k}^{I}(t)=(k+1)\frac{\beta}{\lambda_{1}+\beta}
\Biggl[p_{k}^{E}(t)+\sum_{m=1}^{k}
\biggl( \frac{\lambda_{1}}{\lambda
_{1}+\beta
} \biggr) ^{m}p_{k-m}^{E}(t)
\Biggr].
\end{equation*}
\end{remark}

Consider now Skellam  processes with time change, where the role of
time is
played by the inverse process $Y(t)$ given by (\ref{Y1}).

Let the Skellam process $S(t)$ have parameters $\lambda_{1}$ and
$\lambda
_{2}$ and let us consider the process
\begin{equation}
S_{{I}}(t)=S\bigl(Y(t)\bigr)=N_{1}\bigl(Y(t)
\bigr)-N_{2}\bigl(Y(t)\bigr), \label{SIY}
\end{equation}
where $N_{1}(t)$, $N_{2}(t)$ and $Y(t)$ are independent.

\begin{theorem}\label{theorem7}
Let $S_{{I}}(t)$ be a Skellam process of type I given by
(\ref{SIY}), then the probabilities $r_{k}(t)=P(S_{{I}}(t)=k),k\in Z$,
are given by
\[
r_{k}(t)=\lambda e^{-\beta t} \biggl( \frac{\lambda_{1}}{\lambda
_{2}}
\biggr) ^{\frac{k}{2}}\int_{0}^{\infty}e^{-u ( \lambda_{1}+\lambda
_{2}+\lambda ) }I_{|k|}
( 2u\sqrt{\lambda_{1}\lambda _{2}} ) I_{0} ( 2
\sqrt{\lambda\beta ut} ) du.
\]

The moment generating function is given by
\begin{equation}
\mathsf{E}e^{\theta S_{I}(t)}=\frac{\lambda}{\lambda+f_{s}(-\theta
)\,}%
e^{-\beta tf_{s}(-\theta)(\lambda+f_{s}(-\theta))^{-1}},
\end{equation}
for $\theta$ such that $\lambda+f_{s}(-\theta)\neq0$, where $%
f_{s}(\theta)$ is the Laplace transform of the initial Skellam
process $%
S(t)$.
\end{theorem}

\begin{proof} Using conditioning arguments, we write:
\[
r_{k}(t)=\int_{0}^{\infty}s_{k}(u)h(u,t)du,
\]
and then we insert the expressions for $s_{k}(u)$ and $h(u,t)$, which are
given by formulas (\ref{skellamdistr}) and (\ref{Ydensity}) correspondingly.
The moment generating function can be calculated as follows:
\begin{align*}
\mathsf{E}e^{\theta S_{I}(t)} &=\int_{0}^{\infty}
\mathsf{E}e^{\theta
S(t)}h(u,t)du=\mathsf{E}e^{-f_{s}(-\theta)Y(t)}
\\
&=\frac{\lambda}{\lambda+f_{s}(-\theta)\,}e^{-\beta tf_{s}(-\theta
)(\lambda+f_{s}(-\theta))^{-1}}
\\
&=\frac{\lambda}{\lambda+\lambda_{1}+\lambda_{2}-\lambda
_{1}e^{\theta
}-\lambda_{2}e^{-\theta}\,}\exp \biggl( -\beta t\frac{\lambda
_{1}+\lambda
_{2}-\lambda_{1}e^{\theta}-\lambda_{2}e^{-\theta}}{\lambda+\lambda
_{1}+\lambda_{2}-\lambda_{1}e^{\theta}-\lambda_{2}e^{-\theta}\,
} \biggr)
\end{align*}
for $\theta$ such that $\lambda+\lambda_{1}+\lambda_{2}-\lambda
_{1}e^{\theta}-\lambda_{2}e^{-\theta}\neq0$.
\end{proof}

\begin{remark}\label{remark12}
The expressions for mean, variance and covariance function for the
Skellam process $S(Y(t))$
can be calculated analogously to corresponding calculations for the
process $S(G_N(t))$ (see Remark \ref{remark7}). We obtain:\vadjust{\goodbreak}
\begin{align*}
\mathsf{E}S\bigl(Y(t)\bigr) &= (\lambda_1-\lambda_2 )
\frac{\beta
t+1}{\lambda};\\
\mathsf{Var}S\bigl(Y(t)\bigr) &= (\lambda_1-\lambda_2
)^2\frac{2\beta
t+1}{\lambda^2}+ (\lambda_1+
\lambda_2 )\frac{\beta t+1}{\lambda};\\
\mathsf{Cov}\bigl(S\bigl(Y(t)\bigr),S\bigl(Y(s)\bigr)\bigr) &= (
\lambda_1-\lambda_2 )^2\frac
{2\beta\min(t,s)+1}{\lambda^2}+
(\lambda_1+\lambda_2 )\frac
{\beta\min(t,s)+1}{\lambda}.
\end{align*}
\end{remark}

Consider the time-changed Skellam process of type II:
\begin{equation}
S_{\mathit{II}}(t)=N_{1}\bigl(Y_{1}(t)
\bigr)-N_{2}\bigl(Y_{2}(t)\bigr), \label{SIIY}
\end{equation}
where $Y_{1}(t)$ and $Y_{2}(t)$ are independent copies of the inverse
process $Y(t)$ and independent of $N_{1}(t)$, $N_{2}(t)$.

\begin{theorem}\label{theorem8}
Let $S_{\mathit{II}}(t)$ be the time-changed Skellam process of type II
given by (\ref{SIIY}). Its probability mass function is given by
\begin{align*}
P\bigl(S_{\mathit{II}}(t)=k\bigr) &=e^{-2\beta t}\frac{\lambda^{2}\lambda
_{1}^{k}}{(\lambda
_{1}+\lambda)^{k+1}(\lambda_{2}+\lambda)}\sum
_{n=0}^{\infty}\frac
{ (
\lambda_{1}\lambda_{2} ) ^{n}(n+k+1)(n+1)}{(\lambda_{1}+\lambda
)^{n}(\lambda_{2}+\lambda)^{n}}
\\
&\quad\times\,\mathcal{E}_{1,1}^{n+k+1} \biggl( \frac{\lambda\beta
t}{\lambda
_{1}+\lambda}
\biggr) \mathcal{E}_{1,1}^{n+1} \biggl( \frac{\lambda
\beta t}{%
\lambda_{2}+\lambda}
\biggr)
\end{align*}
for $k\in Z$, $k\geq0$, and when $k<0$
\begin{align*}
P\bigl(S_{\mathit{II}}(t)=k\bigr) &=e^{-2\beta t}\frac{\lambda^{2}\lambda_{2}^{ \llvert
k \rrvert  }}{(\lambda_{1}+\lambda)(\lambda_{2}+\lambda)^{ \llvert
k \rrvert
+1}}\sum
_{n=0}^{\infty}\frac{ ( \lambda_{1}\lambda_{2} )
^{n}(n+k+1)(n+1)}{(\lambda_{1}+\lambda)^{n}(\lambda_{2}+\lambda
)^{n}}
\\
&\quad\times\,\mathcal{E}_{1,1}^{n+1} \biggl( \frac{\lambda\beta t}{\lambda
_{1}+\beta}
\biggr) \mathcal{E}_{1,1}^{n+ \llvert  k \rrvert  +1} \biggl( \frac{%
\lambda\beta t}{\lambda_{2}+\beta}
\biggr) .
\end{align*}
The moment generating function is given by
\begin{align*}
Ee^{\theta S_{\mathit{II}}(t)} &=\frac{\lambda^{2}}{ ( \lambda+\lambda
_{1}(1-e^{-\theta}) )  ( \lambda+\lambda_{2}(1-e^{\theta
}) ) \,}
\\
&\quad\times\exp \biggl( -\beta t\frac{\lambda_{1}(1-e^{-\theta
})}{\lambda
+\lambda_{1}(1-e^{-\theta})\,} \biggr) \exp \biggl( -\beta t
\frac
{\lambda
_{2}(1-e^{\theta})}{\lambda+\lambda_{2}(1-e^{\theta})\,} \biggr)
\end{align*}
for $\theta$ such that $\lambda+\lambda_{2}(1-e^{\theta})\neq
0$.
\end{theorem}

\begin{proof} Analogously to the proof of Theorem \ref{theorem4} we
write the expression for\break $P(S_{\mathit{II}}(t)=k)$ in the form
(\ref{proofSII}) and then we use the expressions for probabilities
$P(N_{1}(Y(t))=k)$ from Theorem \ref{theorem5}. The moment generating function
is obtained as the product: $\mathsf{E}e^{\theta
S_{\mathit{II}}(t)}=\mathsf{E}e^{\theta
N_{1}(Y_{1}(t))}\mathsf{E}e^{-\theta N_{2}(Y_{2}(t))}$, and then
we use expression~(\ref{26}).
\end{proof}

\begin{remark}\label{remark13}
The moments of $S_{\mathit{II}}(t)$ can be calculated using the moment
generating function given in Theorem \ref{theorem8},
or using independence of processes $N_i(Y_i(t)), i=1,2$, and
corresponding expressions for the moments of $N_i(Y_i(t)), i=1,2$.
In view of mutual independence of $N_1(Y_1(t))$ and $N_2(Y_2(t))$ and
using the formula \eqref{covY}, we obtain the covariance function:\vadjust{\goodbreak}
\begin{align*}
\mathsf{Cov} \bigl(S_{\mathit{II}}(t), S_{\mathit{II}}(s) \bigr)&= \mathsf{Cov}
\bigl(N_1\bigl(Y_1(t)\bigr),N_1
\bigl(Y_1(s)\bigr) \bigr)
+\mathsf{Cov} \bigl(N_2\bigl(Y_2(t)\bigr),N_2
\bigl(Y_2(s)\bigr) \bigr)
\\
&=\frac{\lambda_1+\lambda_2}{\lambda} \bigl(\beta\min(t,s)+1 \bigr)+\frac{\lambda_1^2+\lambda_2^2}{\lambda^2} \bigl(2
\beta \min(t,s)+1 \bigr).
\end{align*}
\end{remark}

\subsection{{Inverse compound Poisson--Erlang process}}\label{section5.2}

Consider now the compound Poisson--Erlang process $G_{N}^{ ( n )
}(t) $, that is, the Poisson-Gamma process with $\alpha=n$.

For this case the inverse process
\begin{equation}
Y^{(n)}(t)=\inf\bigl\{u\geq0;\text{ }G_{N}^{ ( n ) }(u)>t
\bigr\},\quad%
t\geq0, \label{Yn}
\end{equation}
has density of the following form:
\begin{equation}
q(s,t)=\lambda e^{-\beta t-\lambda s}\sum_{k=1}^{n}
( \beta t ) ^{k-1}\varPhi \bigl( n,k,\lambda s ( \beta t )
^{n} \bigr) . \label{hYn}
\end{equation}
The formula for the density (\ref{hYn}) (in different set of
parameters) was obtained in \cite{PV} (Theorem 3.1) by developing
the approach presented in \cite{Cox}. This approach is based on
calculating and inverting the Laplace transforms, by taking into
account the relationship between Laplace transforms of a direct
process and its inverse process. We refer for details to
\cite{Cox,PV} (see also Appendix \ref{appendixa.2}). It should be noted that
for the compound
Poisson-Gamma processes with a non-integer parameter $\alpha$ the
inverting Laplace transforms within this approach leads to
complicated infinite integrals (see again \cite{PV}).

Laplace transform of the process $Y^{(n)}(t)$ can be represented in the
following form:
\begin{align*}
\mathsf{E}\bigl[e^{-\theta Y^{(n)}(t)}\bigr] &=e^{-\beta
t}\sum
_{k=1}^{n}\sum_{j=0}^{\infty}
\frac{ ( \beta t )
^{nj+k-1}}{%
\varGamma(nj+k)} \biggl( \frac{\lambda}{\lambda+\theta} \biggr) ^{j+1}
\\
&=e^{-\beta t}\frac{\lambda}{\lambda+\theta}\sum_{k=1}^{n}
( \beta t ) ^{k-1}\mathcal{E}_{n,k} \biggl(
\frac{\lambda ( \beta
t )
^{n}}{\lambda+\theta} \biggr)
\end{align*}
(see (\ref{A.4}) in Appendix \ref{appendixa.2}).
With direct calculations, using the known form of the density of
$Y^{ (
n ) }(t)$, we find the following expressions for the moments:
\begin{align*}
\mathsf{E}Y^{(n)}(t) &=e^{-\beta t}\frac{1}{\lambda}\sum
_{k=1}^{n} ( \beta t ) ^{k-1}
\mathcal{E}_{n,k}^{2} \bigl( ( \beta t ) ^{n}
\bigr) ,
\\
\mathsf{E} \bigl[ Y^{(n)}(t) \bigr] ^{2}
&=e^{-\beta t}\frac
{2}{\lambda^{2}}%
\sum
_{k=1}^{n} ( \beta t ) ^{k-1}
\mathcal{E}_{n,k}^{3} \bigl( ( \beta t ) ^{n}
\bigr) ,
\end{align*}
and, generally, for $p\geq1$:
\[
\mathsf{E} \bigl[ Y^{(n)}(t) \bigr] ^{p}=e^{-\beta t}
\frac{p!}{\lambda
^{p}}%
\sum_{k=1}^{n}
( \beta t ) ^{k-1}\mathcal{E}_{n,k}^{p+1} \bigl( (
\beta t ) ^{n} \bigr)
\]
(see details of calculations in Appendix \ref{appendixa.2}).

\begin{remark}\label{remark14}
Using the arguments similar to those in \cite{PV} (see proof of
Lemma 3.11 therein), we can also derive another expression for the
first moment, in terms of the\vadjust{\goodbreak} two-parameter generalized
Mittag-Leffler function:
\begin{equation}
\mathsf{E}Y^{(n)}(t)=\frac{\beta t}{n\lambda}+\frac{e^{-\beta
t}}{n\lambda}%
\sum_{k=1}^{n} ( n+k-1 ) ( \beta t )
^{k-1}\mathcal {E}%
_{n,k} \bigl( ( \beta t )
^{n} \bigr) . \label{EYn}
\end{equation}
From (\ref{EYn}) we can see that $\mathsf{E}Y^{(n)}(t)$ has a linear behavior
with respect to $t$ as $t\rightarrow\infty$:
\begin{equation}
\mathsf{E}Y^{(n)}(t)\sim\frac{\beta t}{n\lambda
}+\frac{n+1}{2n\lambda},
\label{EYnasymp}
\end{equation}
%
which should be indeed expected, since the general result holds for
subordinators with finite mean: the mean of their first passage time exhibit
linear behavior for large times (see, for example, \cite{VT} and references
therein).

The details of derivation of (\ref{EYn}) and (\ref{EYnasymp}) are
presented in Appendix \ref{appendixa.2}.
\end{remark}

Consider the time-changed process $Z^{ ( n )
}(t)=N_{1}(Y^{ (
n ) }(t))$, where $Y^{ ( n ) }(t)$ is the inverse process
given by (\ref{Yn}), independent of $N_{1}(t)$.

\begin{theorem}\label{theorem9}
The probability mass function of the process $Z^{ ( n )
}(t)=N_{1}(Y^{ ( n ) }(t))$ is given by
\begin{equation}
p_{k}(t)=P\bigl(Z^{ ( n ) }(t)=k\bigr)=e^{-\beta t}
\frac{\lambda\lambda
_{1}^{k}(k+1)}{(\lambda_{1}+\lambda)^{k+1}}\sum_{m=1}^{n} ( \beta t
) ^{m-1}\mathcal{E}_{n,m}^{k+1} \biggl(
\frac{\lambda ( \beta
t ) ^{n}}{\lambda_{1}+\lambda} \biggr) ,
\end{equation}
Laplace transform is given by
\[
E \bigl[ e^{-\theta Z^{ ( n ) }(t)} \bigr] =e^{-\beta t}\frac{%
\lambda}{\lambda+\lambda_{1}(1-e^{-\theta})\,}\sum
_{k=1}^{n} ( \beta t ) ^{k-1}
\mathcal{E}_{n,k} \biggl( \frac{\lambda ( \beta
t )
^{n}}{\lambda+\lambda_{1}(1-e^{-\theta})} \biggr) .
\]
\end{theorem}

\begin{proof} \xch{Proof is similar}{is similar} to that for Theorem \ref{theorem5}. In
particular, the probability mass function is obtained as follows:
\begin{align*}
\hspace*{20pt}p_{k}(t) &=P\bigl(N_{1}\bigl(Y^{ ( n ) }(t)\bigr)=k
\bigr)
\\
&=\int_{0}^{\infty}e^{-\lambda_{1}s}
\frac{ ( \lambda_{1}s
) ^{k}%
}{k!}\lambda e^{-\lambda s-\beta t}\sum_{m=1}^{n}
( \beta t ) ^{m-1}\varPhi \bigl( n,m,\lambda s ( \beta t )
^{n} \bigr) ds
\\
&=\int_{0}^{\infty}e^{-\lambda_{1}s-\lambda s-\beta t}\lambda
\frac{
 ( \lambda_{1}s ) ^{k}}{k!}\sum_{m=1}^{n} ( \beta t
) ^{m-1}\sum_{l=0}^{\infty}
\frac{ ( \lambda s ( \beta t )
^{n} ) ^{l}}{l!\varGamma ( nl+m ) }ds
\\
&=\frac{e^{-\beta t}\lambda\lambda_{1}^{k}}{k!}\sum_{m=1}^{n} (
\beta t ) ^{m-1}\sum_{l=0}^{\infty}
\frac{ ( \lambda ( \beta
t ) ^{n} ) ^{l}}{l!\varGamma ( nl+m ) }\int_{0}^{\infty
}e^{-(\lambda_{1}+\lambda)s}s^{k+l}ds
\\
&=e^{-\beta t}\frac{\lambda\lambda_{1}^{k}(k+1)}{(\lambda
_{1}+\lambda
)^{k+1}}\sum_{m=1}^{n}
( \beta t ) ^{m-1}\mathcal{E}%
_{n,m}^{k+1}
\biggl( \frac{\lambda ( t\beta ) ^{n}}{\lambda
_{1}+\lambda} \biggr) .\hspace*{40pt}\qedhere
\end{align*}
\end{proof}

\begin{remark}\label{remark15}
The first two moments of the process $Z^{ ( n ) }(t)$ can be
calculated as follows:
\begin{align*}
\mathsf{E}Z^{ ( n ) }(t)&=\mathsf{E}N_{1}(1)
\mathsf{E}%
Y^{(n)}(t)=e^{-\beta t}\frac{\lambda_{1}}{\lambda}\sum
_{k=1}^{n} ( \beta t ) ^{k-1}
\mathcal{E}_{n,k}^{2} \bigl( ( \beta t ) ^{n}
\bigr) ,\\
\mathsf{E} \bigl[ Z^{ ( n ) }(t) \bigr] ^{2}&=
\mathsf{Var}%
N_{1}(1)U(t)- \bigl[ \mathsf{E}N_{1}(1)
\bigr] ^{2}\mathsf{E} \bigl[ Y(t) 
 \bigr] ^{2}\\
&=e^{-\beta t}\frac{\lambda_{1}}{\lambda}\sum_{k=1}^{n}
( \beta t ) ^{k-1}\mathcal{E}_{n,k}^{2} \bigl( (
\beta t ) ^{n} \bigr) -e^{-\beta t}\frac{\lambda_{1}^{2}}{\lambda}\sum
_{k=1}^{n} ( \beta t ) ^{k-1}
\mathcal{E}_{n,k}^{3} \bigl( ( \beta t ) ^{n}
\bigr) ,
\end{align*}
and we can see that, similarly to $\mathsf{E}Y^{(n)}(t)$, $\mathsf{E}%
Z^{ ( n ) }(t)$ has linear behavior as $t\rightarrow\infty$:
\[
\mathsf{E}Y^{(n)}(t)\sim\frac{\lambda_{1}\beta t}{n\lambda
}+\frac{\lambda_{1} ( n+1 ) }{2n\lambda}.
\]
\end{remark}

Let the Skellam process $S(t)$ have parameters $\lambda_{1}$ and
$\lambda
_{2}$ and let us consider the process
\begin{equation}
S_{{I}}^{ ( n ) }(t)=S\bigl(Y^{ ( n ) }(t)
\bigr)=N_{1}\bigl(Y^{ (
n ) }(t)\bigr)-N_{2}
\bigl(Y^{ ( n ) }(t)\bigr), \label{SIYn}
\end{equation}
where $N_{1}(t)$, $N_{2}(t)$ and $Y^{ ( n ) }(t)$ are independent.

\begin{theorem}\label{theorem10}
Let $S_{{I}}^{ ( n ) }(t)$ be a Skellam process of type I
given by (%
\ref{SIYn}), then the probabilities $r_{k}(t)=P(S_{{I}}^{ ( n )
}(t)=k)$, $k\in Z$, are given by
\begin{align*}
r_{k}(t) &=\lambda e^{-\beta t} \biggl( \frac{\lambda_{1}}{\lambda
_{2}}%
 \biggr) ^{\frac{k}{2}}\sum_{m=1}^{n} (
\beta t ) ^{m-1}\int_{0}^{\infty}e^{-u ( \lambda_{1}+\lambda_{2}+\lambda
 ) }I_{|k|}
( 2u\sqrt{\lambda_{1}\lambda_{2}} )
\\
&\quad\times\varPhi \bigl( n,m,\lambda u ( \beta t ) ^{n} \bigr) du.
\end{align*}

The moment generating function is given by
\[
E \bigl[ e^{\theta S_{{I}}^{ ( n ) }(t)} \bigr] =e^{-\beta
t}\frac{%
\lambda}{\lambda+f_{s}(-\theta)\,}\sum
_{k=1}^{n} ( \beta t ) ^{k-1}
\mathcal{E}_{n,k} \biggl( \frac{\lambda ( \beta t )
^{n}}{%
\lambda+f_{s}(-\theta)} \biggr)
\]
for $\theta$ such that $\lambda+f_{s}(-\theta)\neq0$, where $%
f_{s}(\theta)$ is the Laplace transform of the initial Skellam
process $%
S(t)$.
\end{theorem}

\begin{proof} Proof is analogous to that of Theorem \ref{theorem7}.
\end{proof}

Consider the time-changed Skellam process of type II:
\begin{equation}
S_{\mathit{II}}^{ ( n ) }(t)=N_{1}\bigl(Y_{1}(t)
\bigr)-N_{2}\bigl(Y_{2}(t)\bigr), \label{SIIYn}
\end{equation}
where $Y_{1}(t)$ and $Y_{2}(t)$ are independent copies of the inverse
process $Y^{ ( n ) }(t)$ and independent of $N_{1}(t)$, $N_{2}(t)$.

\begin{theorem}\label{theorem11}
Let $S_{\mathit{II}}^{ ( n ) }(t)$ be the time-changed Skellam
process of type II given by (\ref{SIIYn}). Its probability mass
function is given by
\begin{align*}
P\bigl(S_{\mathit{II}}^{ ( n ) }(t)=k\bigr) &=e^{-2\beta t}
\frac{\lambda
^{2}\lambda
_{1}^{k}}{(\lambda_{1}+\lambda)^{k+1}(\lambda_{2}+\lambda)}%
\sum_{p=0}^{\infty}
\frac{ ( \lambda_{1}\lambda_{2} )
^{p}(p+k+1)(n+1)}{(\lambda_{1}+\lambda)^{p}(\lambda_{2}+\lambda
)^{p}}
\\
&\quad\times\sum_{m_{1}=1}^{n}\sum
_{m_{2}=1}^{n} ( \beta t ) ^{m_{1}+m_{2}-2}\,
\mathcal{E}_{n,m_{1}}^{p+k} \biggl( \frac{\lambda (
\beta t ) ^{n}}{\lambda_{1}+\lambda} \biggr)
\mathcal{E}%
_{n,m_{2}}^{p} \biggl( \frac{\lambda ( \beta t ) ^{n}}{\lambda
_{2}+\lambda}
\biggr)
\end{align*}
for $k\in Z$, $k\geq0$, and when $k<0$
\begin{align*}
P\bigl(S_{\mathit{II}}^{ ( n ) }(t)=k\bigr) &=e^{-2\beta t}
\frac{\lambda
^{2}\lambda
_{2}^{ \llvert  k \rrvert  }}{(\lambda_{1}+\lambda)(\lambda_{2}+\lambda
)^{ \llvert  k \rrvert  +1}}\sum_{p=0}^{\infty}
\frac{ ( \lambda
_{1}\lambda
_{2} ) ^{p}(p+k+1)(p+1)}{(\lambda_{1}+\lambda)^{p}(\lambda
_{2}+\lambda)^{p}}
\\
&\quad\times\sum_{m_{1}=1}^{n}\sum
_{m_{2}=1}^{n} ( \beta t ) ^{m_{1}+m_{2}-2}\,
\mathcal{E}_{n,m_{1}}^{p} \biggl( \frac{\lambda (
\beta
t ) ^{n}}{\lambda_{1}+\lambda} \biggr)
\mathcal {E}_{n,m_{2}}^{p+|k|}%
 \biggl(
\frac{\lambda ( \beta t ) ^{n}}{\lambda_{2}+\lambda} 
 \biggr) .
\end{align*}
The moment generating function is
\begin{align*}
E \bigl[ e^{\theta S_{\mathit{II}}(t)} \bigr] &=e^{-2\beta t}\frac{\lambda
}{\lambda
+\lambda_{1}(1-e^{\theta})\,}\sum
_{k=1}^{n} ( \beta t ) ^{k-1}%
\mathcal{E}_{n,k} \biggl( \frac{\lambda ( \beta t )
^{n}}{\lambda
+\lambda_{1}(1-e^{\theta})} \biggr)
\\
&\quad\times\frac{\lambda}{\lambda+\lambda_{2}(1-e^{-\theta})\,}%
\sum_{k=1}^{n}
( \beta t ) ^{k-1}\mathcal{E}_{n,k} \biggl(
\frac
{%
\lambda ( \beta t ) ^{n}}{\lambda+\lambda_{2}(1-e^{-\theta
})}%
 \biggr)
\end{align*}
for $\theta$ such that $\lambda+\lambda_{1}(1-e^{\theta})\neq
0$.
\end{theorem}

\begin{proof} Proof is analogous to that of Theorem \ref{theorem8}.
\end{proof}

\begin{remark}\label{remark16}
Covariance structure of the Skellam processes considered in this
section appears to be of complicated form and we postpone this
issue for future research.
\end{remark}

\section*{Acknowledgments}
The authors are grateful to the referees for their valuable comments
and suggestions which helped to improve the paper.

\appendix
\section{Appendix}

\subsection{{Proof of Lemma \ref{lemma1}: calculation of the
covariance function
of inverse compound Poisson-exponential process}}\label{appendixa.1}

Let $E_{N}(t)$ be the compound Poisson-exponential process with
parameters $%
\lambda,\beta$, that is, with Laplace exponent $f(u)=\lambda\frac{u}{
\beta+u}$. Consider the first passage time of $E_{N}(t)$:
\[
Y(t)=\inf\bigl\{u\geq0;E_{N}(u)>t\bigr\},\quad t\geq0.
\]

First two moments of $Y(t)$ and $\mathsf{Var} Y(t)$ are presented in
(\ref
{momentsY}) and can be directly calculated using the probability
density of $%
Y(t)$ given by (\ref{Ydensity}).

To calculate the covariance we need to find the expression for $\mathsf
{E}%
Y(t_{1})Y(t_{2})$.

We will use Theorem 3.1 from \cite{VT} which gives the expressions for the
Laplace transforms for $n$-th order moments of the inverse process in
terms of
Laplace transforms of lower-order moments.

Denote
\[
U(t_{1},t_{2},m_{1},m_{2})=
\mathsf{E}Y(t_{1})^{m_{1}}Y(t_{2})^{m_{2}},
\]
where $m_{1}$ and $m_{2}$ are positive integers; denote also $\mathsf{E}
Y(t_{1})=U(t_{1})$.

Let $\tilde{U}(u_{1},u_{2},m_{1},m_{2})$ be the Laplace transform of $%
U(t_{1},t_{2},m_{1},m_{2})$. Then, in these notations,
\[
\mathsf{E}Y(t_{1})Y(t_{2})=U(t_{1},t_{2},1,1)
\]
and from Theorem 3.1, formula (17) \cite{VT} we have:
\begin{align*}
\tilde{U}(u_{1},u_{2},1,1) &=\frac{1}{f(u_{1}+u_{2})} \bigl(
\tilde{U}%
(u_{1},u_{2},1,0)+\tilde{U}(u_{1},u_{2},0,1)
\bigr)
\\
&=\frac{1}{f(u_{1}+u_{2})} \biggl( \frac{\tilde{U}(u_{1})}{u_{1}}+\frac{
\tilde{U}(u_{2})}{u_{2}} \biggr) .
\end{align*}
In the above formula $\tilde{U}(u_{1},u_{2},1,0)$ is the Laplace transform
of $U(t_{1},t_{2},1,0)=\mathsf{E}Y(t_{1})=U(t_{1})$ and $\tilde{U}%
(u_{1},u_{2},0,1)$ is the Laplace transform of
$U(t_{1},t_{2},0,1)=\mathsf{E}%
Y(t_{2})=U(t_{2})$.

The inverse Laplace transform can be found by the following calculations:
\begin{align}
&U(t_{1},t_{2},1,1)\notag\\
&\quad=\mathcal{L}^{-1} \bigl(
\tilde {U}(u_{1},u_{2},1,1) \bigr) (t_{1},t_{2})\notag\\
&\quad=\int_{0}^{t_{1}}\int_{0}^{t_{2}}
\bigl[ U(t_{1}-\tau_{1})+U(t_{2}-\tau
_{2})%
 \bigr] \mathcal{L}^{-1} \biggl(
\frac{1}{f(u_{1}+u_{2})} \biggr) (\tau _{1},\tau_{2})d
\tau_{1}d\tau_{2}; \label{*}
\end{align}
for the function
\[
\frac{1}{f(u_{1}+u_{2})}=\frac{\beta+u_{1}+u_{2}}{\lambda
(u_{1}+u_{2})}=%
\frac{\beta}{\lambda}
\frac{1}{u_{1}+u_{2}}+\frac{1}{\lambda}
\]
we write the inverse Laplace transform in the form
\begin{equation}
\mathcal{L}^{-1} \biggl( \frac{1}{f(u_{1}+u_{2})} \biggr) (
\tau_{1},\tau _{2})d\tau_{1}d\tau_{2}=
\frac{\beta}{\lambda}\delta(\tau_{1}-\tau _{2})d
\tau_{1}d\tau_{2}+\frac{1}{\lambda}\delta(
\tau_{1})\delta (\tau _{2})d\tau_{1}d
\tau_{2}, \label{**}
\end{equation}
and continue calculations inserting (\ref{**}) in (\ref{*}):
\begin{align*}
U(t_{1},t_{2},1,1) &=\frac{\beta}{\lambda}\int
_{0}^{t_{1}}\int_{0}^{t_{2}}%
 \bigl[ U(t_{1}-\tau_{1})+U(t_{2}-
\tau_{2}) \bigr] \delta(\tau _{1}-\tau _{2})d
\tau_{1}d\tau_{2}
\\
&\quad+\frac{1}{\lambda}\int_{0}^{t_{1}}\int
_{0}^{t_{2}} \bigl[ U(t_{1}-\tau
_{1})+U(t_{2}-\tau_{2}) \bigr] \delta(
\tau_{1})\delta(\tau _{2})d\tau _{1}d
\tau_{2}
\\
&=\frac{\beta}{\lambda}\int_{0}^{t_{1}\wedge t_{2}} \bigl[
U(t_{1}-\tau) +U(t_{2}-\tau) \bigr] d\tau+
\frac{1}{\lambda
} \bigl( U(t_{1})+U(t_{2}) \bigr) .
\end{align*}

Therefore, for the covariance of the process $Y(t)$ we obtain the following
expression:
\begin{align*}
\mathsf{Cov} \bigl( Y(t_{1}),Y(t_{2}) \bigr) &=
\frac{\beta}{\lambda}%
\int_{0}^{t_{1}\wedge t_{2}} \bigl[
U(t_{1}-\tau)d\tau+U(t_{2}-\tau ) \bigr] d\tau
\\
&\quad+\frac{1}{\lambda} \bigl( U(t_{1})+U(t_{2}) \bigr)
-U(t_{1})U(t_{2}).
\end{align*}


Using the expression for $U(t)=\frac{1}{\lambda}(\beta t+1)$, we
find:
\begin{align*}
\mathsf{E} \bigl( Y(t_{1})Y(t_{2}) \bigr)
=U(t_{1},t_{2},1,1)&=\frac
{\beta
}{\lambda^{2}}\bigl(\beta (
t_{1}+t_{2} ) +2\bigr)\min(t_{1},t_{2})
\\
&\quad-\frac{\beta^{2}}{\lambda^{2}} \bigl( \min(t_{1},t_{2}) \bigr)
^{2}+%
\frac{1}{\lambda^{2}}\bigl(\beta ( t_{1}+t_{2}
) +2\bigr)
\end{align*}
and
\begin{align*}
&\mathsf{Cov} \bigl( Y(t_{1}),Y(t_{2}) \bigr)\\
&\quad=\mathsf{E}
\bigl[ Y(t_{1})Y(t_{2}) \bigr] -\mathsf{E}Y(t_{1})
\mathsf{E}Y(t_{2})
\\
&\quad=\frac{\beta}{\lambda^{2}}\bigl(\beta ( t_{1}+t_{2} ) +2\bigr)
\min (t_{1},t_{2})-\frac{\beta^{2}}{\lambda^{2}} \bigl( \min
(t_{1},t_{2}) \bigr) ^{2}-\frac{\beta^{2}}{\lambda
^{2}}t_{1}t_{2}+
\frac{1}{%
\lambda^{2}}
\\
&\quad=\frac{1}{\lambda^{2}}\bigl(2\beta\min(t_{1},t_{2})+1
\bigr).
\end{align*}

\subsection{{Marginal distribution and moments of the
process $Y^{(n)}(t)$}}\label{appendixa.2}

We present some details of the derivation of the expression for
probability density of the
inverse Poisson--Erlang process $Y^{(n)}(t)$ introduced by the formula
(\ref{Yn}) in Section~\ref{section5.2}.

The inverse Poisson--Erlang process was considered in \cite{PV} and its
probability
density function (p.d.f.) was presented in Theorem 3.1 therein.
However, in \cite{PV} the different parametrization of the
Poisson--Erlang process was used in comparison with that used in
our paper.

For convenience of a reader and to make the paper
self-contained,
we present here some details of calculations following the general
approach developed in \cite{Cox}.

Introduce the Laplace transforms related to the process $
G_{N}^{(n)}(t)$ and its inverse process $Y^{(n)}(t)$:
\begin{align*}
^{*}l(v,t)&=\mathsf{E}e^{-v G_{N}^{(n)}(t)};\\
^{*}l^{*}(v,u)&=\int_{0}^{\infty}
{^{*}l}(v,t)e^{-ut}dt
\end{align*}
and
\begin{align*}
^{*}q(v,t)&=\mathsf{E}e^{-v Y^{(n)}(t)};\\
^{*}q^{*}(v,u)&=\int_{0}^{\infty}
{ ^{*}q}(v,t)e^{-ut}dt.
\end{align*}
Then the following relation holds (see \cite{Cox}, Section~8.4,
formula (3)):
\begin{eqnarray}
^{*}q^{*}(v,u)= \frac{1-v\, ^{*}l^{*}(u,v)}{u}. \label{A. 1}
\end{eqnarray}

The above formula holds, in fact, for more general compound
Poisson processes. In the case of compound Poisson process with
jumps having the p.d.f. $g(x)$, it is possible to write the exact
expression for $^{*}l^{*}(v,u)$ and formula (\ref{A. 1}) takes the
form:\vadjust{\goodbreak}
\begin{eqnarray}
^{*}q^{*}(v,u)= \frac{\hat{f}(v) (1-\hat{g}(u) )}{u (1-\hat{g}(u)\hat
{f}(v) )}, \label{A.2}
\end{eqnarray}
where $f(x)$ is the p.d.f. of exponential distribution,
$\hat{f}(v)$ and $\hat{g}(u)$ are the Laplace transforms of $f$
and $g$ correspondingly (see formula (4) Section~8.4 in
\cite{Cox}).

For the case of Poisson--Erlang process, $f(x)$ is the p.d.f. of
exponential distribution and $g(x)$ is the p.d.f. of Erlang
distribution. Inserting the expressions for $\hat{f}(v)$ and
$\hat{g}(u)$ in (\ref{A.2}) we finally obtain:
\begin{eqnarray}
^{*}q^{*}(v,u)= \frac{\lambda ( (\beta+u )^n-\beta^n
)}{u (\lambda+v ) (\beta+u )^n-\lambda\beta^n}. \label{A.3}
\end{eqnarray}

One special case when inversion of (\ref{A.3})
can be easily performed was considered in \cite{Cox}, namely, the
case when $f$ and $g$ are both exponential, and in such a case we
come to the p.d.f. of inverse compound Poisson-exponential process
(see formula (\ref{Ydensity}) in Section~\ref{section5.1}).
The exact result is also available for the inverse Poisson--Erlang
process. This result
was stated in Theorem 3.1 of \cite{PV}, for its proof the
inverse Laplace transforms for \eqref{A.3} were calculated
consequently with respect to variables $v$ and $u$.

Here we present a reverse check, namely, we check that the double
Laplace transform of the p.d.f. (\ref{hYn}) gives the expression
(\ref{A.3}). We have:
\begin{align}
^{*}q(v,t)&= \int_{0}^{\infty}e^{-vs}
\lambda e^{-\beta t - \lambda s} \sum_{k=1}^{n} (
\beta t )^{k-1}\sum_{m=0}^{\infty}
\frac{
(\lambda s  (\beta t  )^n )^m}{m! \varGamma(nm+k)}ds
\nonumber
\\
&=\lambda e^{-\beta t}\sum_{k=1}^{n}
(\beta t )^{k-1} \sum_{m=0}^{\infty}
\frac{  (\beta t  )^{nm}\lambda^m}{m! \varGamma
(nm+k)}\int_{0}^{\infty}e^{-(v+\lambda)s}s^m
ds
\nonumber
\\
&=e^{-\beta
t}\sum_{k=1}^{n}\sum
_{m=0}^{\infty}\frac{ ( \beta t )
^{nm+k-1}}{%
\varGamma(nm+k)} \biggl(
\frac{\lambda}{\lambda+v } \biggr) ^{m+1} \label{A.4}
\end{align}
and
\begin{align*}
^{*}q^{*}(v,u)&= \biggl( \frac{\lambda}{\lambda+v } \biggr)
^{m+1}\int_{0}^{\infty}e^{-\beta
t}\sum
_{k=1}^{n}\sum_{m=0}^{\infty}
\frac{ ( \beta t )
^{nm+k-1}}{(nm+k-1)!}e^{-ut}dt
\\
&=\sum_{k=1}^{n}\sum
_{m=0}^{\infty}\frac{ ( \beta )
^{nm+k-1}}{(nm+k-1)!} \biggl(
\frac{\lambda}{\lambda+v } \biggr) ^{m+1}\int_{0}^{\infty}e^{-(\beta+u)t}t^{nm+k-1}dt
\nonumber
\\
&=\sum_{k=1}^{n}\frac{\lambda}{v+\lambda}
\frac{\beta^{k-1}}{
(\beta+u )^k}\sum_{m=0}^{\infty} \biggl(
\frac{\lambda\beta^n}{
(v+\lambda ) (\beta+u )^n} \biggr)^m
\nonumber
\\
&=\frac{\lambda
 ( (\beta+u )^n-\beta^n )}{u (\lambda+v
) (\beta+u )^n-\lambda
\beta^n},
\end{align*}
which coincides indeed with (\ref{A.3}).

We next obtain the expressions for the moments of the process
$Y^{(n)}(t)$. Using the known form of the probability density of
the process $Y^{(n)}(t)$, we obtain:
\begin{align*}
\mathsf{E}Y^{(n)}(t) &=\int_{0}^{\infty}s
\lambda e^{-\beta t-\lambda
s}\sum_{k=1}^{n}(
\beta t)^{k-1}\varPhi\bigl(n,k,\lambda s(\beta t)^{n}\bigr)
\\
&=\int_{0}^{\infty}s\lambda e^{-\beta t-\lambda s}\sum
_{k=1}^{n}(\beta t)^{k-1}\sum
_{l=0}^{\infty}\frac{ ( \lambda s(\beta t)^{n} )
^{l}}{%
l!\varGamma(nl+k)}ds
\\
&=\lambda e^{-\beta t}\sum_{k=1}^{n}(
\beta t)^{k-1}\sum_{l=0}^{\infty
}%
\frac{ ( \lambda(\beta t)^{n} ) ^{l}}{l!\varGamma(nl+k)}%
\int_{0}^{\infty}e^{-\lambda s}s^{l+1}ds
\\
&=\lambda e^{-\beta t}\sum_{k=1}^{n}(
\beta t)^{k-1}\sum_{l=0}^{\infty
}%
\frac{ ( \lambda(\beta t)^{n} ) ^{l}\varGamma(l+2)}{l!\varGamma
(nl+k)\lambda^{l+2}}
\\
&=\frac{1}{\lambda}e^{-\beta t}\sum_{k=1}^{n}(
\beta t)^{k-1}\mathcal {E}%
_{n,k}^{2}
\bigl( (\beta t)^{n} \bigr) ,
\end{align*}
and, analogously, for the moment of the general order $p\geq1$:
\begin{align*}
\mathsf{E} \bigl[ Y^{(n)}(t) \bigr] ^{p} &=\int
_{0}^{\infty
}s^{p}\lambda e^{-\beta t-\lambda s}
\sum_{k=1}^{n}(\beta t)^{k-1}
\varPhi\bigl(n,k,\lambda s(\beta t)^{n}\bigr)
\\
&=\int_{0}^{\infty}s^{p}\lambda
e^{-\beta t-\lambda s}\sum_{k=1}^{n}(\beta
t)^{k-1}\sum_{l=0}^{\infty}
\frac{ ( \lambda s(\beta t)^{n} )
^{l}}{%
l!\varGamma(nl+k)}ds
\\
&=\lambda e^{-\beta t}\sum_{k=1}^{n}(
\beta t)^{k-1}\sum_{l=0}^{\infty
}%
\frac{ ( \lambda(\beta t)^{n} ) ^{l}}{l!\varGamma(nl+k)}%
\int_{0}^{\infty}e^{-\lambda s}s^{l+p}ds
\\
&=\frac{p!}{\lambda^{p}}e^{-\beta t}\sum_{k=1}^{n}(
\beta t)^{k-1}\mathcal{E%
}_{n,k}^{p+1}
\bigl( (\beta t)^{n} \bigr) .
\end{align*}
To find the moments we can also use the moment generating function of $%
Y^{(n)}(t)$:
\[
\mathsf{E}\bigl[e^{\theta Y^{(n)}(t)}\bigr]=e^{-\beta t}\frac{\lambda}{\lambda
-\theta}\sum
_{k=1}^{n} ( \beta t ) ^{k-1}
\mathcal {E}_{n,k} \biggl( \frac{\lambda ( \beta t ) ^{n}}{\lambda-\theta} \biggr)
\]
(with argument $\theta<\lambda)$; differentiating with respect to
$\theta$
and taking the derivative at $\theta=0$, we obtain the expectation of $
Y^{(n)}(t)$ in the following form:
\begin{align*}
\mathsf{E}Y^{(n)}(t) &=\frac{e^{-\beta t}}{\lambda}\sum
_{k=1}^{n}%
\sum
_{j=0}^{\infty}\frac{(j+1) ( \beta t ) ^{nj+k-1}}{\varGamma
(nj+k)%
}
\\
&=\frac{e^{-\beta t}}{\lambda}\sum_{k=1}^{n} (
\beta t ) ^{k-1}\sum_{j=0}^{\infty}
\frac{j ( \beta t )
^{nj}}{\varGamma(nj+k)} +\frac{e^{-\beta t}}{\lambda}\sum_{k=1}^{n}
( \beta t ) ^{k-1}%
\mathcal{E}_{n,k} \bigl( ( \beta
t ) ^{n} \bigr) .
\end{align*}
Using some calculations from \cite{PV} (see proof of Lemma 3.11
therein), we
can represent the first term in the above expression in the following form:
\[
\frac{\beta t}{n\lambda}-\frac{e^{-\beta t}}{n\lambda}\sum_{k=1}^{n}
( k-1 ) ( \beta t ) ^{k-1}\mathcal{E}_{n,k} \bigl( ( \beta t
) ^{n} \bigr) ,
\]
and, therefore, we obtain
\[
\mathsf{E}Y^{(n)}(t)=\frac{\beta t}{n\lambda}+\frac{e^{-\beta
t}}{n\lambda}%
\sum_{k=1}^{n} ( n-k+1 ) ( \beta t )
^{k-1}\mathcal {E}%
_{n,k} \bigl( ( \beta t )
^{n} \bigr) ;
\]
and then we can use the asymptotic relation
\[
\mathcal{E}_{n,k} \bigl( ( \beta t ) ^{n} \bigr) \sim
\frac
{1}{n}%
 ( \beta t ) ^{1-k}e^{\beta t}\quad\text{as}\ t\rightarrow \infty,
\]
to come to the following asymptotics for $\mathsf{E}Y^{(n)}(t)$:
\[
\mathsf{E}Y^{(n)}(t)\sim\frac{\beta t}{n\lambda}+\frac{n+1}{2n\lambda
}%
\quad\text{as}\ t\rightarrow\infty.
\]

\end{document}